\documentclass[12pt, a4paper]{amsart}

\usepackage{amsmath}

\usepackage{amssymb}
\usepackage{stmaryrd}%for \longmapsfrom
\usepackage{amsthm}
\usepackage{mathrsfs}

\usepackage{amscd}
\usepackage[all]{xy}
\usepackage{url}

\newtheorem{Thm}{Theorem}[subsection]
\usepackage{comment}%\end{comment} must be used in the starting of the
\usepackage[dvipsnames, usenames]{color}
\usepackage[colorlinks]{hyperref}
\hypersetup{%
pdftitle={Preprint},
pdfauthor={Fan QIN},
pdfkeywords={Quantum groups, quiver varieties, categorification,
  dual canonical basis},
bookmarksnumbered,
pdfstartview={FitH},
breaklinks=true,
urlcolor=blue,
citecolor=blue,
}%
\usepackage[all]{hypcap} %to fix fraction numbers showing problems in package {hyperref}

\usepackage{breakurl}%If there are problems in the references (which might occur if use dvi->ps->pdf), use the following package

\usepackage{color}

\newtheorem{Lem}[Thm]{Lemma}
\newtheorem{Prop}[Thm]{Proposition}

\newtheorem{Eg}[Thm]{Example}
\newtheorem{Rem}[Thm]{Remark}
\newtheorem{Def}[Thm]{Definition}

\newtheorem*{Def*}{Definition}
\newtheorem*{Thm*}{Theorem}

\newcommand{\Z}{\mathbb{Z}}
\newcommand{\N}{\mathbb{N}}
\newcommand{\Q}{\mathbb{Q}}
\newcommand{\C}{\mathbb{C}}

\newcommand{\cf}{{\em cf.}\ }

\newcommand{\resp}{{resp.}\ }
\renewcommand{\hat}[1]{\widehat{#1}}
\renewcommand{\tilde}[1]{\widetilde{#1}}

\newcommand{\opname}[1]{\operatorname{\mathsf{#1}}}
\renewcommand{\mod}{\opname{mod}}%redefined!

\newcommand{\Rep}{\opname{Rep}}

\newcommand{\pr}{\opname{pr}}

\newcommand{\ind}{\opname{ind}}

\newcommand{\proj}{\opname{proj}}
\newcommand{\inj}{\opname{inj}}

\newcommand{\Inj}{\opname{Inj}}

\newcommand{\op}{^{op}}

\newcommand{\ra}{\rightarrow}
\newcommand{\xra}{\xrightarrow}

\newcommand{\Ker}{\opname{Ker}}
\newcommand{\Cok}{\opname{Cok}}

\newcommand{\Hom}{\opname{Hom}}

\newcommand{\Ext}{\opname{Ext}}

\renewcommand{\deg}{\opname{deg}}

\newcommand{\Hf}{{\frac{1}{2}}}
\newcommand{\xla}{\xleftarrow}

\newcommand{\Rm}[1]{{\longmapsto}}
\newcommand{\Lm}[1]{{\longmapsfrom}}

\newcommand{\cA}{{\mathcal A}}

\newcommand{\cD}{{\mathcal D}}

\newcommand{\cF}{{\mathcal F}}

\newcommand{\cI}{{\mathcal I}}

\newcommand{\cL}{{\mathcal L}}
\newcommand{\cM}{{\mathcal M}}
\newcommand{\cN}{{\mathcal N}}

\newcommand{\cR}{{\mathcal R}}
\newcommand{\cS}{{\mathcal S}}

\newcommand{\cW}{{\mathcal W}}
\newcommand{\cX}{{\mathcal X}}

\newcommand{\bI}{{\mathbf I}}

\newcommand{\bL}{{\mathbf L}}

\newcommand{\ua}{{\underline{a}}}

\newcommand{\ui}{{\underline{i}}}

\newcommand{\tB}{{\widetilde{B}}}

\newcommand{\can}{L}%simple module, dual canonical basis
\newcommand{\redWSet}{{\mathcal{J}}}%coefficient Free W-vectors

\newcommand{\tBase}{{\Z[t^\pm]}}

\newcommand{\projQuot}{{\cM}}%projective quotient
\newcommand{\affQuot}{{\projQuot_0}}%affine quotient
\newcommand{\grProjQuot}{{\projQuot^\bullet}}
\newcommand{\grAffQuot}{{\affQuot^\bullet}}

\newcommand{\cycProjQuot}{{\projQuot^\epsilon}}
\newcommand{\cycAffQuot}{{\affQuot^\epsilon}}

\newcommand{\grRegStratum}{{\grAffQuot^\mathrm{reg}}}%fiber of projective $\pi$
\newcommand{\cycRegStratum}{{\cycAffQuot^\mathrm{reg}}}%fiber of projective $\pi$
\newcommand{\qClAlg}{{\cA^q}}

\newcommand{\oh}{{\overline{h}}}

\newcommand{\diag}{{\delta}}
\newcommand{\RHS}{{\mathrm{RHS}}} %right hand side
\newcommand{\LHS}{{\mathrm{LHS}}}

\newcommand{\tRes}{{\tilde{\mathrm{Res}}}}
\newcommand{\res}{{\mathrm{Res}}}

\newcommand{\trunc}{{^{\leq 0}}}

\newcommand{\tw}{{\tilde{w}}}

\newcommand{\tOtimes}{\tilde{\otimes}}

\newcommand{\mfr}[1]{{\mathfrak{#1}}}
\newcommand{\mbf}[1]{{\mathbf{#1}}}

\newcommand{\msf}[1]{{\mathsf{#1}}}
   
\newcommand{\braket}[1]{\left\langle#1\right\rangle}
\newcommand{\set}[1]{\left\{#1\right\}}

\newcommand{\Uq}{\mbf{U}_{q}}

\newcommand{\Ut}{\mbf{U}_{t}}
\newcommand{\tUt}{\tilde{\mbf{U}}_{t}}

\newcommand{\tKappa}{\tilde{\kappa}}
\usepackage{tikz}
\usetikzlibrary{positioning,shapes,shadows,arrows,snakes}

\pgfdeclarelayer{edgelayer}
\pgfdeclarelayer{nodelayer}
\pgfsetlayers{edgelayer,nodelayer,main}

\tikzstyle{none}=[inner sep=0pt]
\tikzstyle{black box}=[draw=black, fill=black!25]
\tikzstyle{white box}=[draw=black, fill=white]
\tikzstyle{black circle}=[circle,draw=black!50, fill=black!25]
\tikzstyle{red circle}=[circle,draw=red!50, fill=red!25]
\tikzstyle{blue circle}=[circle,draw=blue!50, fill=blue!25]
\tikzstyle{green circle}=[circle,draw=green!50, fill=green!25]
\tikzstyle{yellow circle}=[circle,draw=yellow!50, fill=yellow!25]

\begin{document}
%\title[]{}
\title[]{Quantum groups via cyclic quiver varieties I}
\author{Fan QIN}
%\email{qin@math.unistra.fr}%add
% \address[]{l'Institut de Recherche Math\'ematique Avanc\'ee (IRMA)\\
% 7 rue Ren\'e Descartes\\
% 67084 Strasbourg Cedex\\
% France}

%\classification{16G20, 17B37 (primary).}
\keywords{Quantum groups, quiver varieties, categorification,
  dual canonical basis}

\begin{abstract}
We construct the
quantized enveloping algebra of any simple Lie algebra of type $\mathbb{A}$$\mathbb{D}$$\mathbb{E}$ as the quotient of a Grothendieck ring arising from certain cyclic quiver varieties. In
particular, the dual canonical basis of a one-half quantum group with respect to
Lusztig's bilinear form is contained in the natural basis of the
Grothendieck ring up to rescaling.

This paper expands the categorification established by
Hernandez and Leclerc to the whole quantum groups. It can be viewed as a geometric counterpart of Bridgeland's recent
work for type $\mathbb{A}$$\mathbb{D}$$\mathbb{E}$.

% We also give an appendix to explain how to translate the main
% construction to the graded quiver varieties, which should be related to
% quasi derived Hall algebras.

% In particular, the dual canonical basis of a half quantum group can be identified as a subset of a natural basis of the Grothendieck ring. 
\end{abstract}
\maketitle
%\end{document}

\tableofcontents
%%%%%%%%%%%%%%%%%%%%%%%%%%
%%                      Main content
%%%%%%%%%%%%%%%%%%%%%%%%%%

%%%%%%%%%%%%%%%%%%%%%%%special macros in this paper%%%%%%%%%
\renewcommand{\qClAlg}{{\cA}}
\renewcommand{\diag}{{d}}

\newcommand{\redISet}{\cI}
\newcommand{\qNilAlg}{\cN}
\newcommand{\norm}{\alpha}
\newcommand{\sym}{\mathcal{S}}
\newcommand{\mHf}{{-\Hf}}
\newcommand{\height}{{\opname{ht}}}

\renewcommand{\inj}{{\bI}}
\renewcommand{\can}{{\bL}}
\newcommand{\Ind}{{\opname{Ind}}}

\newcommand{\wtMap}{{\opname{wt}}}
\newcommand{\gMap}{{\opname{wt}^{-1}}}

\renewcommand{\redWSet}{{\cW}}
\newcommand{\degL}{\mathrm{D}}
\newcommand{\domDegL}{\mathrm{D}^\dagger}

\newcommand{\condL}{\mathrm{L}}
%%%%%%%%%%%%%%%%%%%%%%%%
%Quantum group

\newcommand{\roots}{{\mbox{\boldmath $r$}}}
\newcommand{\bpi}{{\mbox{\boldmath $\pi$}}}

\newcommand{\hfg}{\hat{\mathfrak{g}}}
\renewcommand{\trunc}{^{\leq 2l}}

%%%%%%%%%%%%%%%%%%%%

\section{Introduction}
\label{sec:intro}

\subsection{History}
\label{sec:background}

For any given symmetric
Cartan datum, let $\mfr{g}$ be the associated Kac-Moody Lie algebra and $\Ut(\mfr{g})$ the
corresponding quantized enveloping algebra. There have been several
different approaches to the categorical realizations of $\Ut(\mfr{g})$. 

\subsubsection*{One-half quantum group}
\label{sec:one_half}

The earliest and best developed theories are categorifications of
a one-half quantum group. Notice that $\Ut(\mfr{g})$ has the triangular decomposition
$\Ut(\mfr{g})=\Ut(\mfr{n}^+)\otimes \Ut(\mfr{h})\otimes
\Ut(\mfr{n}^-)$. Let $Q$ denote a quiver associated with $\mfr{g}$
which has no oriented cycles. For any field $k$, let $k Q$ denote the
path algebra associated with $Q$.

1) In 1990, Ringel showed in \cite{Ringel90} that the positive (resp. negative) one-half
quantum group $\Ut(\mfr{n}^+)$ (resp. $\Ut(\mfr{n}^-)$) can be realized as a subalgebra of
the\emph{ Hall algebra} of the abelian category $\mathbb{F}_q Q-\mod$, where $\mathbb{F}_q$
is any finite field and $\mathbb{F}_q Q-\mod$ the category of the left
modules of the path algebra $\mathbb{F}_q Q$. Let us call this Hall algebra approach an \emph{additive categorification} of
$\Ut(\mfr{n}^+)$, because the product of any Chevalley generator with
itself is translated into the direct sum of a simple $\mathbb{F}_q Q$-module
with itself.

2) Lusztig has given a geometric construction of
$\Ut(\mfr{n}^+)$, \cf \cite{Lusztig90} \cite{Lusztig91}, by considering the
Grothendieck ring arising from certain perverse sheaves over the
varieties of $\C Q$-modules. This geometric approach is very
powerful. In particular, the perverse sheaves provides us a positive basis\footnote{By
a positive basis, we mean a basis whose structure constants are non-negative.} of
$\Ut(\mfr{n}^+)$, which is called \emph{the canonical basis}, \cf also
\cite{Kas:crystal} for the \emph{crystal basis}. We can view Lusztig's construction as a
\emph{monoidal categorification} of $\Ut(\mfr{n}^+)$, because the addition
and the multiplication in $\Ut(\mfr{n}^+)$ are translated into the direct
sum and the derived tensor of perverse sheaves.

We recommend the reader to the survey papers of Schiffmann \cite{Schiffmann06}
\cite{Schiffmann:canonical_basis} for the results 1) and 2).

3) Recently, the quiver Hecke algebras (or KLR-algebras, \cf
\cite{KhovanovLauda08} \cite{Rouquier08}) provide us a monoidal categorification of the one-half quantum group
$\Ut(\mfr{n}^+)$. A link between Lusztig's approach and the
quiver Hecke algebras has been established in \cite{VaragnoloVasserot09}.

4) Finally, assume that $\mfr{g}$ is of type $\mathbb{ADE}$. Hernandez and Leclerc showed that certain
subcategory of finite dimensional representations of the quantum
affine algebra $\Uq(\hat{\mfr{g}})$ provides a monoidal
categorification of $\Ut(\mfr{n}^+)$, \cf \cite{HernandezLeclerc11}. By \cite{Nakajima01} and
\cite[section 9]{HernandezLeclerc11}, their construction can be
understood in terms of graded quiver varieties and then be compared with the
work of Lusztig.

We remark that the categorification in 3) and 4) are compatible with
the (dual) canonical basis obtained in 2). Moreover, by 4), the categorification in present paper is compatible with the
dual canonical basis.
\subsubsection*{Whole quantum group}

We can define the algebra $\tUt(\mfr{g})$ as a variant of the
whole quantum group $\Ut(\mfr{g})$, \cf section \ref{sec:quantum_group}, which has
the triangle decomposition $\Ut(\mfr{n}^+)\otimes \tUt(\mfr{h})\otimes
\Ut(\mfr{n}^-)$. This variant plays a crucial role in Bridgland's work
 \cite{Bridgeland:Hall}, which we shall briefly recall. The whole
quantum group $\Ut(\mfr{g})$ is obtained from $\tUt(\mfr{g})$ by a
reduction at the Cartan part $\tUt(\mfr{h})$.

There have been various attempts to make a Hall algebra construction of the whole quantum group, \cf for example \cite{Kapranov98},
\cite{PengXiao97}, \cite{PengXiao00} \cite{XiaoXuZhang06}. The complete result
was obtained in the recent work of Bridgeland.

\begin{Thm*}[\cite{Bridgeland:Hall}]
Fix a finite field $\mathbb{F}_q$. Let $\tilde{U}_{\sqrt{q}}(\mfr{g})[(K_i)^{-1},(K_i')^{-1}]_{i\in I}$ denote the localization of $\tilde{U}_{\sqrt{q}}(\mfr{g})$ at the
Cartan part. Then it is isomorphic to the localization
of the Ringel Hall algebra of the $2$-periodic
complexes of projective $\mathbb{F}_q Q$-modules at the contractible complexes.
\end{Thm*}
The usual quantum group $\Ut(\mfr{g})$ can be obtained from the above
construction as the natural quotient of $\tUt(\mfr{g})[(K_i)^{-1},(K_i')^{-1}]_{i\in I}$.

In the work of Bridgeland, the realizations of the half-quantum groups $\Ut(\mfr{n}^+)$ and
$\Ut(\mfr{n}^-)$ can be identified with those in Ringel's
approach. The Cartan part $\tUt(\mfr{h})$ is generated by certain
complexes homotopic to zero, which are redundant information in the
study of the corresponding triangulated category. In the sense of section \ref{sec:one_half},
this Hall algebra approach can be viewed as an additive
categorification of $\Ut(\mfr{g})$.

Also, by the works of Khovanov, Lauda, Rouquier, and Ben Webster, \cf
\cite{KhovanovLauda08:III} \cite{Rouquier08} \cite{Webster10}
\cite{Webster13}, the quiver Hecke algebras provide a monoidal categorification of the modified quantum group
$\dot{\mbf{U}}_t(\mfr{g})$, which is a different variant of the whole
quantum group $\Ut(\mfr{g})$ \cite{Lus:intro}.

Finally, we notice that Fang and Rosso have constructed the whole quantum group in
the spirit of quantum shuffle algebras, \cf \cite{FangRosso12}.

\subsection{Main construction and result}
\label{sec:main_result}

In this paper, we give a geometric construction of the
whole quantum group for the Lie algebra $\mfr{g}$ of Dynkin type $\mathbb{A}$,
$\mathbb{D}$, $\mathbb{E}$, inspired by the following papers.

Inspired by \cite[Theorem 2.7]{KellerScherotzke2013}, we use some cyclic
quiver varieties associated
with roots of unity to replace the abelian category of
$2$-periodic complexes in \cite{Bridgeland:Hall}.

Then, the work of Hernandez and Leclerc \cite[Section
9]{HernandezLeclerc11} establishes a construction of the half-quantum groups $\Ut(\mfr{n}^+)$ and
$\Ut(\mfr{n}^-)$, which can be compared with Lusztig's
work by \cite[Section 9]{HernandezLeclerc11}. The techniques developed in \cite{HernandezLeclerc11} and
\cite{LP_rep_alg} will be crucial in our proofs.

We construct the Cartan part $\tUt(\mfr{h})$ from certain
strata of cyclic quiver varieties, which are identified with the
stratum $\{0\}$ in Nakajima's \emph{transverse slice theorem}
\cite[3.3.2]{Nakajima01}. The analog of these strata for graded quiver
varieties provides redundant information in
the study of quantum affine algebras \cite{{Nakajima01}}. So our
construction of the Cartan part shares the same spirit as that of
Bridgeland's work : 

\begin{center}
  Cartan part is categorified by redundant information.
\end{center}

In the sense of our previous discussion in categorification 2),
this geometric construction can be viewed as a monoidal categorification
of $\tUt(\mfr{g})$, which contains the Hernandez-Leclerc categorification of
$\Ut(\mfr{n})$. In particular, we obtain a positive basis of $\tUt(\mfr{g})$, which, up
to rescaling, contains the dual canonical basis of a one-half
quantum group with respect to Lusztig's bilinear form.

We refer the reader to Section
\ref{sec:construction} for the detailed construction and Theorems
\ref{thm:isom} \ref{thm:Lusztig_basis} for the rigorous statements of the results.

\subsection{Remarks}

This paper could be viewed as a geometric counterpart of Bridgeland's
work for the type $\mathbb{ADE}$. It is natural to compare this geometric construction
with the Hall algebra construction of Bridgeland. The details might appear
elsewhere. 

On the other hand, by choosing the shifted simple modules in derived
categories as in \cite[Section 8.2]{HernandezLeclerc11}, the analogous construction in the present paper
remains effective over graded quiver varieties associated with a
generic $q$. Details might appear elsewhere. The corresponding Grothendieck ring should then
be compared with the semi-derived Hall algebra associated with the
quiver $Q$ in
the sense of Gorsky \cite{Gorsky13}. However, this straightforward
generalization is not the unique approach. A completely different construction might
appear in Gorsky's future work.

In this paper, the twisted product defined for the
Grothendieck ring is different from those used by
\cite{Nakajima04} or by \cite{Hernandez02}\cite{HernandezLeclerc11}. It is
worth mentioning that our twisted product agrees with the non-commutative multiplication of
\cite{Hernandez02}\cite{HernandezLeclerc11} on the one half quantum group
$\Ut(\mathfrak{n}^+)$, as we shall prove in the last
section. We refer the reader to Example \ref{eg:compare_product} for a
comparison of various products.

Our Grothendieck ring $\tUt(\msf{g})$ are defined over some cyclic
quiver varieties, which are closely related
to the Grothendieck ring of finite dimensional representations of
the quantum affine algebra $\Uq(\hat{\msf{g}})$ at the root of unity
$q$. In particular, in \cite{Nakajima04}, Nakajima has used these cyclic
quiver varieties to study the $t$-analog of the $q$-characters on the
latter Grothendieck ring. However, to the best knowledge of the author, there exists no twisted product
in literature such that the
Cartan part of $\tUt(\msf{g})$ consists of center elements, which prevents a
direct reduction of $\tUt(\msf{g})$ to the $t$-deformed Grothendieck ring of
representations of $\Uq(\hat{\msf{g}})$ considered in
\cite{Nakajima04}\cite{Hernandez04}.

Finally, the present paper is just a first step of this geometric
approach. In particular, the reduction of the Cartan part discussed
here follows a straightforward algebraic approach, which was used by
Bridgeland. We shall use the corresponding geometric realization to
study quantum groups in a future work.

\section{Preliminaries}
\label{sec:preliminaries}

\subsection{Quantum groups}
\label{sec:quantum_group}
We recall the basic facts concerning
the quantum groups and refer the reader to Schiffmann's
note \cite{Schiffmann06} or Lusztig's book \cite{Lus:intro} for more details. We
shall follow the notations used in \cite{KimuraQin11} \cite{Kimura10}.

Let $n$ be any given positive integer and define the index set $I=\{1,\ldots,n\}$. Fix a symmetric
root datum. Denote the Cartan matrix by $C=(a_{ij})_{i,j\in I}$ and
the positive simple roots by $\set{\alpha_i,i\in I}$. Let $\mfr{g}$
be the corresponding Kac-Moody Lie algebra.

% Let $v$ be an indeterminate, which should not be confused with the dimension
% vector $v$ we will use when studying the quiver varieties. 

Let $t$ be an indeterminate. We define $[n]_t =(t^n-t^{-n})/(t-t^{-1})$, $[n]_t!=[1]_t\cdot [2]_t\cdot\ldots \cdot[n]_t$. Let $\tUt(\mfr{g})$ be the $\Q(t)$-algebra generated by the \emph{Chevalley generators} $E_i,
K_i,K_i',F_i$, $i\in I$, which are subject to the following
relations
\begin{align*}
  \sum_{k=0}^{1-a_{ij}}(-1)^kE_i^{(k)}E_jE_i^{(1-a_{ij}-k)}=0,\\
\sum_{k=0}^{1-a_{ij}}(-1)^kF_i^{(k)}F_jF_i^{(1-a_{ij}-k)}=0,\\
  [E_i,F_j]=\delta_{ij}\frac{K_i-K_i'}{t-t^{-1}},\\
[K_i,K_j]=[K_i,K_j']=[K_i',K_j']=0,
\end{align*}
\begin{align*}
K_iE_j=t^{a_{ij}}E_jK_i,\\
K_iF_j=t^{-a_{ij}}F_jK_i,\\
K'_iE_j=t^{-a_{ij}}E_jK'_i,\\
K'_iF_j=t^{a_{ij}}F_jK'_i,
\end{align*}
where $E_i^{(k)}=E_i^k/[k]_t!$ and $F_i^{(k)}=F_i^k/[k]_t!$.

The \emph{quantum group} $\Ut(\mfr{g})$ is defined as the quotient algebra of $\tUt(\mfr{g})$ with
respect to the ideal generated by the elements $K_i *
K_i'-1$, $i\in I$.

Let $\tUt(\mfr{n}^+)$ be the subalgebra of $\tUt(\mfr{g})$ generated
by $E_i$, $i\in I$, $\tUt(\mfr{h})$ the subalgebra of $\tUt(\mfr{g})$
generated by $K_i$, $K_i'$, $\tUt(\mfr{n}^-)$ the subalgebra of $\tUt(\mfr{g})$ generated
by $F_i$. The subalgebras $\Ut(\mfr{n}^+)$, $\Ut(\mfr{h})$,
$\Ut(\mfr{n}^-)$ of $\Ut(\mfr{g})$ are defined similarly. Then both $\tUt(\mfr{g})$ and $\Ut(\mfr{g})$ have triangular decompositions:
\begin{align*}
  \tUt(\mfr{g})&=\tUt(\mfr{n}^+)\otimes \tUt(\mfr{h})\otimes \tUt(\mfr{n}^-),\\
  \Ut(\mfr{g})&=\Ut(\mfr{n}^+)\otimes \Ut(\mfr{h})\otimes \Ut(\mfr{n}^-).
\end{align*}
From the definitions, we have $\Ut(\mfr{n}^+)=\tUt(\mfr{n}^+)$, $\Ut(\mfr{n}^-)=\tUt(\mfr{n}^-)$, and $\Ut(\mfr{h})=\tUt(\mfr{h})/(K_i*K_i'-1)_i$.

% Recall that $\Ut(\mfr{g})$ has a $\Q$-algebra anti-involution $\Omega$
% such that for any $i\in I$, we have
% \begin{align*}
% \Omega(v)=t^{-1},\\
%   \Omega(E_i)=F_i,\\
% \Omega(F_i)=E_i,\\
% \Omega(K_i)=K_i^{-1}.
% \end{align*}

The \emph{Kashiwara's bilinear form} $(\ ,\ )_K$ on $\Ut(\mfr{n}^+)$
has the property $(E_i,E_j)_K=\delta_{ij}$, \cf \cite[Section 3.4]{Kas:crystal}.
The \emph{Lusztig's bilinear form} $(\ ,\ )_L$ on $\Ut(\mfr{n}^+)$ has
the property $(E_i,E_j)_L=\delta_{ij}(1-t^2)^{-1}$, \cf
\cite[1.2.5]{Lus:intro}. In general, by \cite[2.2]{Leclerc04}, for any homogeneous
elements $x,y\in\Ut(\mfr{n}^+)_{\beta}$, where $\beta=\sum_{i\in I}
\beta_i \alpha_i$, $\beta_i\in\N$, we
have
\begin{align}\label{eq:compare_bilinear_form}
  (x,y)_K=(1-t^2)^{\sum_i \beta_i}\cdot (x,y)_L.
\end{align}

We let $A_t(\mfr{n}^+)$ denote the quantum coordinate ring which is
the graded dual $\Q(t)$-vector space of $\Ut(\mfr{n}^+)$ endowed with a
restricted multiplication, \cf \cite[section 4]{GeissLeclercSchroeer11}
and also \cite[section 3]{Kimura10}. % By the following proposition, we
% will not distinguish the two algebras but only use the notion $\Ut(\mfr{n}^+)$.

\begin{Prop}[{\cite[Proposition 4.1]{GeissLeclercSchroeer11}}]\label{prop:dual_alg}
  There exists an algebra isomorphism $\Psi$ from $\Ut(\mfr{n}^+)$ to $A_t(\mfr{n}^+)$ such
  that any element $x$ is sent to the linear map $(x,\ )_K$.
\end{Prop}
% Notice that we can identify $(E_i,\ )_K$ with $E_i$ via $(\ ,\ )_K$.

\subsection{Graded and cyclic quiver varieties}
\label{sec:quiver_variety}
In this paper, we consider the quivers\footnote{The quiver
  $Q$ used in this paper should be compared with the opposite quiver $Q\op$ used in
  \cite{KimuraQin11}} $Q$ of type $\mathbb{A}$, $\mathbb{D}$, $\mathbb{E}$.

Choose any $q\in\C^*$ such that $q\neq 1$. It generates a
cyclic subgroup $\langle q \rangle$ in the multiplicative group
 $(C^*,*)$. We assume that either $\langle q\rangle$ is an infinite
 group or its cardinality is divisible by $2$.

Because the underlying graph of the Dynkin quiver $Q$ is a tree, we
can choose a height function $\xi:I\ra \langle q\rangle$ such that
$\xi(k)=q*\xi(i)$ whenever there is an arrow from $k$ to $i$ in $Q$. 

Define $\hat{I}=\{(i,a)\in I\times \langle q \rangle |
\xi(i)*a^{-1}\in \langle q^2 \rangle \}$. The reader is referred to
Example \ref{eg:A_3_quiver} for an example.

Let $\sigma$ denote the automorphism of $I\times \langle q \rangle$
such that $\sigma(i,a)=(i,q^{-1}a)$. Then  $I\times \langle q \rangle$
is the disjoint union of $\hat{I}$ and $\sigma\hat{I}$. We use $\tau$ to denote the automorphism $\sigma^2$ on $I\times
\langle q \rangle$.

We always use $x$ to denote the elements in
$\sigma\hat{I}$. We use $v$, $w$ to denote the finitely supported elements in
$\N^{\sigma\hat{I}}$, $\N^{\hat{I}}$ respectively. Let $e_{i,a}$
denote the characteristic function of $(i,a)$, which is also viewed as the
unit vector supported at $(i,a)$. We have $\sigma^*e_{\sigma
  (i,a)}=e_{i,a}$. For any given $v,w$, we denote the associated
$I\times \braket{q}$-graded vector spaces by $V=\oplus_{i,a}
V(i,a)=\oplus \C^{v(i,a)}$ and $W=\oplus_{i,a}
W(i,a)=\oplus \C^{w(i,a)}$.

The \emph{q-Cartan matrix} $C_q$ is a linear map from $\Z^{\sigma\hat{I}}$,
$\Z^{\hat{I}}$, such that for any $(i,a)\in\sigma\hat{I}$, we have
\begin{align}
  C_q e_{i,a}= e_{i,q a}+e_{i,q^{-1}a}+\sum_{j\in I, j\neq i}a_{ij} e_{j,a}.
\end{align}

A pair $(v,w)$ is called \emph{$l$-dominant} if $w-C_qv\geq 0.$

We shall define
graded/cyclic quiver varieties. Details could be found in \cite{Nakajima01} (\cf also \cite{Nakajima09} \cite{Qin12}
\cite{KimuraQin11}).

Let $\Omega$ denote the set of the arrows of $Q$. Similarly, let
$\overline{\Omega}$ denote the set
of the arrows of the opposite quiver $Q\op$. For each arrow $h$, we
let $s(h)$ and $t(h)$ denote its source and target respectively. Define
\begin{align}
  E^q(\Omega;v,w)&=\oplus_{(i,a)\in \sigma\hat{I}}\oplus_{h\in \Omega
    :s(h)=i,t(h)=j}\Hom(V(i,a),V(j,aq^{-1})),\\
L^q(w,v)&=\oplus_{x\in \sigma\hat{I}}\Hom(W(\sigma ^{-1}x),V( x)),\\
L^q(v,w)&=\oplus_{x\in \sigma\hat{I}}\Hom(V(x),W(\sigma x)).
\end{align}
Define the vector space $\Rep^q(Q;v,w)$ to be
\begin{align}
  \Rep^q(Q;v,w)=E^q(\Omega;v,w)\oplus E^q(\overline{\Omega};v,w)\oplus
  L^q(w,v)\oplus L^q(v,w),
\end{align}
whose elements are denoted by 
\begin{align*}
(\oplus_hB_h,\oplus_{\oh}B_{\oh},(\alpha_i)_{i\in I},(\beta_i)_{i\in I})=(\oplus_{h\in \Omega}(\oplus_a B_{h,a}),\oplus_{\oh\in
    \overline{\Omega}}(\oplus_b B_{\oh,b}),(\oplus_a \alpha_{i,a})_{i\in
    I},(\oplus_b\beta_{i,b})_{i\in I})
\end{align*}
The group $GL_v=\prod_{x\in \sigma \hat{I}}GL_{v(x)}$ naturally acts on
$\Rep^q(Q;v,w)$. We have the map $\mu$ as the natural analog of
the moment map such that
\begin{align*}
  \mu((\oplus_hB_h,\oplus_{\oh}B_{\oh},(\alpha_i),(\beta_i))=&\sum_{h\in
    \Omega, \oh'\in\overline{\Omega}:s(h)=t(\oh')}B_hB_{\oh'}-\sum_{h\in
    \Omega, \oh'\in\overline{\Omega}:s(\oh')=t(h)}B_{\oh'}B_h+\sum_{i\in
  I}\alpha_i\beta_i,
\end{align*}
\cf \cite{Nakajima01} for details.

For the $GL_v$-variety $\mu^{-1}(0)$, we construct
Mumford's GIT\footnote{GIT stands for ``Geometric invariant theory''.} quotient $\mathcal{M}^q(v,w)$ and the categorical quotient $\mathcal{M}^q_0(v,w)$. There is a natural proper morphism $\pi$ from the
GIT quotient $\mathcal{M}^q(v,w)$ to the categorical quotient $\mathcal{M}^q_0(v,w)$.

\begin{Eg}\label{eg:A_3_quiver}
  Let the quiver $Q$ be given by Figure \ref{fig:A_3_quiver}. We
  can choose the height function $\xi$ such that
  $\xi(i)=q^{i-1}$. Then the $I\times \langle q
  \rangle$ is given by Figure \ref{fig:A_3_I}, where the blue vertices
  (in squared box)
  belong to $\hat{I}$ the other vertices belong to $\sigma \hat{I}$.

Then the vector space
  $\Rep^q(Q;v,w)$ is described in Figure \ref{fig:A_3_rep_space}, whose
rows and columns are indexed by $I$-degrees (vertices) and
$\braket{q}$-degrees (heights)
respectively. 

In this example, the analog of the moment map $\mu$ take the form
\begin{align*}
  (\alpha_1\beta_1+B_{h_1}B_{\overline{h}_1})\oplus
  (\alpha_2\beta_2+B_{h_2}B_{\overline{h}_2}-B_{\overline{h_1}}B_{h_1})\oplus (\alpha_3\beta_3-B_{\overline{h}_2}B_{h_2}).
\end{align*}
\end{Eg}

\begin{figure}[htb!]
 \centering
\beginpgfgraphicnamed{A_3_quiver}
  \begin{tikzpicture}
    \node [] (1) at (0,-2) {1}; 
    \node  [] (2) at (2,0) {2}; 
    \node [] (3) at (4,2) {3};
    \draw[-stealth] (3) edge node[above] {$h_2$} (2); 
    \draw[-stealth] (2) edge node[above] {$h_1$} (1);
  \end{tikzpicture}
\endpgfgraphicnamed
\caption{A quiver of type $A_3$}
\label{fig:A_3_quiver}
\end{figure}
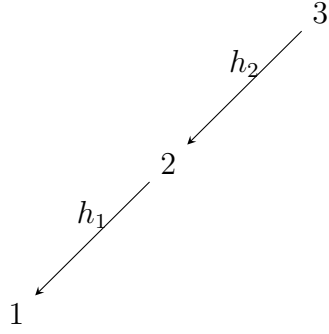

\begin{figure}[htb!]
 \centering
\beginpgfgraphicnamed{A_3_I}
\begin{tikzpicture}[scale=0.55]

\node [draw,color=blue] (w10) at (-12,-6) {$(1,1)$};
\node [draw,color=blue] (w21) at (-8,-2) {$(2,q)$};
\node [draw,color=blue] (w32) at (-4,2) {$(3,q^2)$};
\node [draw,color=blue] (w12) at (-4,-6) {$(1,q^2)$};
\node [draw,color=blue] (w23) at (0,-2) {$(2,q^3)$};
\node [draw,color=blue] (w34) at (4,2) {$(3,q^4)$};

\node [color=blue] (w14) at (4,-6) {$\cdots$};
\node [color=blue] (w25) at (8,-2) {$\cdots$};

\node [] (v11) at (-8,-6) {$(1,q)$};
\node [] (v22) at (-4,-2) {$(2,q^2)$};
\node [] (v33) at (0,2) {$(3,q^3)$};
\node [] (v13) at (0,-6) {$(1,q^3)$};
\node [] (v24) at (4,-2) {$(2,q^4)$};
\node [] (v35) at (8,2) {$(3,q^5)$};

%\node [color=blue] (w30) at (-12,2) {$W_3(1)$};
%\node [color=blue] (w16) at (12,-6) {$W_1(q^6)$};
%\node [] (v20) at (-12,-2) {$V_2(1)$};
%\node [] (v31) at (-8,2) {$V_3(q)$};
%\node [] (v15) at (8,-6) {$V_1(q^5)$};
%\node [] (v26) at (12,-2) {$V_2(q^6)$};
\node [] (v20) at (-12,-2) {$\cdots$};
\node [] (v31) at (-8,2) {$\cdots$};
% \node [] (v15) at (8,-6) {$\cdots$};
% \node [] (v26) at (12,-2) {$\cdots$};

% \node [color=blue] (deg0) at (-12,4) {$\mathrm{height}=1$};
% \node [color=blue] (deg1) at (-8,4) {$\mathrm{height}=q$};
% \node [color=blue] (deg2) at (-4,4) {$\mathrm{height}=q^2$};
% \node [color=blue] (deg3) at (0,4) {$\mathrm{height}=q^3$};
% \node [color=blue] (deg4) at (4,4) {$\mathrm{height}=q^4$};
% \node [color=blue] (deg5) at (8,4) {$\mathrm{height}=q^5$} ;
% \node [color=blue] (deg6) at (12,4) {$\mathrm{height}=q^6$} ;
\draw[-latex][dashed] (w12) edge node[above] {$\sigma$}  (v11);
\draw[-latex][dashed] (v13) edge node[above] {$\sigma$}  (w12);

\end{tikzpicture}
\endpgfgraphicnamed
\caption{Vertices $I\times \langle q\rangle$}
\label{fig:A_3_I}
\end{figure}
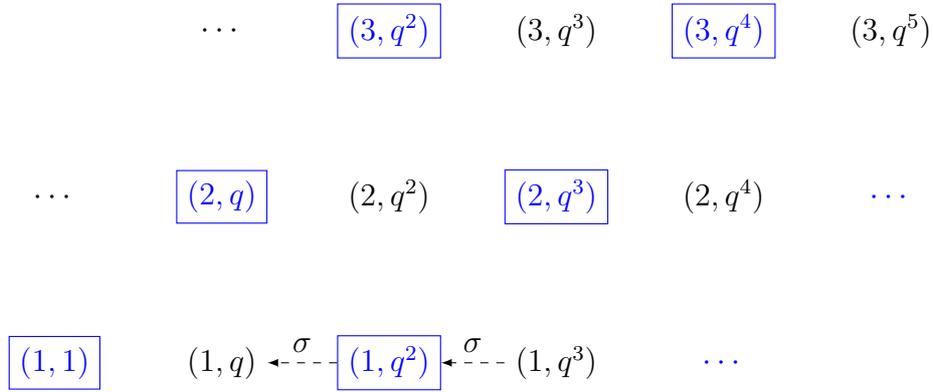

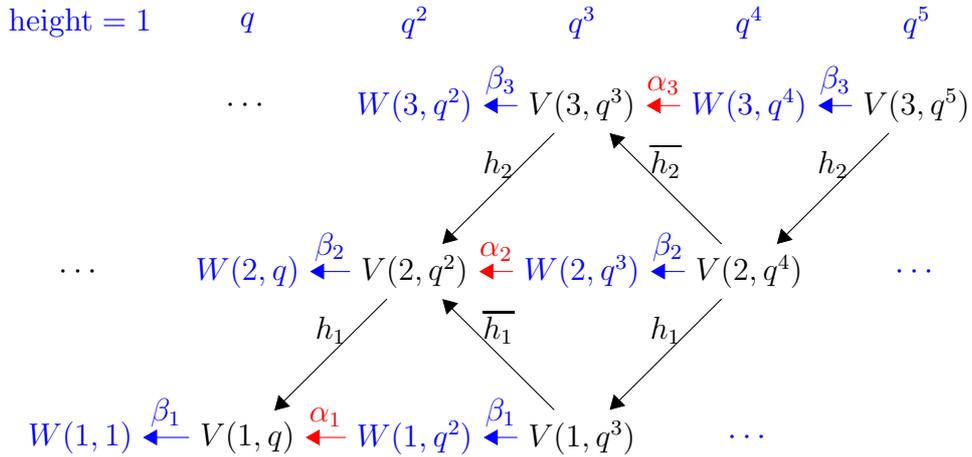
\begin{figure}[htb!]
 \centering
\beginpgfgraphicnamed{A_3_rep_space}
\begin{tikzpicture}[scale=0.55]

\node [color=blue] (w10) at (-12,-6) {$W(1,1)$};
\node [color=blue] (w21) at (-8,-2) {$W(2,q)$};
\node [color=blue] (w32) at (-4,2) {$W(3,q^2)$};
\node [color=blue] (w12) at (-4,-6) {$W(1,q^2)$};
\node [color=blue] (w23) at (0,-2) {$W(2,q^3)$};
\node [color=blue] (w34) at (4,2) {$W(3,q^4)$};

% \node [color=blue] (w14) at (4,-6) {$W(1,q^4)$};
% \node [color=blue] (w25) at (8,-2) {$W(2,q^5)$};
% \node [color=blue] (w36) at (12,2) {$W(3,q^6)$};
\node [color=blue] (w14) at (4,-6) {$\cdots$};
\node [color=blue] (w25) at (8,-2) {$\cdots$};

\node [] (v11) at (-8,-6) {$V(1,q)$};
\node [] (v22) at (-4,-2) {$V(2,q^2)$};
\node [] (v33) at (0,2) {$V(3,q^3)$};
\node [] (v13) at (0,-6) {$V(1,q^3)$};
\node [] (v24) at (4,-2) {$V(2,q^4)$};
\node [] (v35) at (8,2) {$V(3,q^5)$};

%\node [color=blue] (w30) at (-12,2) {$W_3(1)$};
%\node [color=blue] (w16) at (12,-6) {$W_1(q^6)$};
%\node [] (v20) at (-12,-2) {$V_2(1)$};
%\node [] (v31) at (-8,2) {$V_3(q)$};
%\node [] (v15) at (8,-6) {$V_1(q^5)$};
%\node [] (v26) at (12,-2) {$V_2(q^6)$};
\node [] (v20) at (-12,-2) {$\cdots$};
\node [] (v31) at (-8,2) {$\cdots$};
% \node [] (v15) at (8,-6) {$\cdots$};
% \node [] (v26) at (12,-2) {$\cdots$};

\node [color=blue] (deg0) at (-12,4) {$\mathrm{height}=1$};
\node [color=blue] (deg1) at (-8,4) {$q$};
\node [color=blue] (deg2) at (-4,4) {$q^2$};
\node [color=blue] (deg3) at (0,4) {$q^3$};
\node [color=blue] (deg4) at (4,4) {$q^4$};
\node [color=blue] (deg5) at (8,4) {$q^5$} ;
% \node [color=blue] (deg6) at (12,4) {$\mathrm{height}=q^6$} ;

% \draw[-triangle 60][color=red] (w14) edge  node[above] {$\alpha_1$} (v13);
\draw[-triangle 60][color=red] (w12) edge  node[above] {$\alpha_1$} (v11);

% \draw[-triangle 60][color=red] (w25) edge  node[above] {$\alpha_2$} (v24);
\draw[-triangle 60][color=red] (w23) edge  node[above] {$\alpha_2$} (v22);

% \draw[-triangle 60][color=red] (w36) edge  node[above] {$\alpha_3$} (v35);
\draw[-triangle 60][color=red] (w34) edge  node[above] {$\alpha_3$} (v33);

\draw[-triangle 60][color=blue] (v13) edge node[above] {$\beta_1$} (w12);
\draw[-triangle 60][color=blue] (v11) edge node[above] {$\beta_1$} (w10);

\draw[-triangle 60][color=blue] (v24) edge node[above] {$\beta_2$} (w23);
\draw[-triangle 60][color=blue] (v22) edge node[above] {$\beta_2$} (w21);

\draw[-triangle 60][color=blue] (v35) edge node[above] {$\beta_3$} (w34);
\draw[-triangle 60][color=blue] (v33) edge node[above] {$\beta_3$} (w32);

\draw[-triangle 60] (v24) edge node[above] {$h_1$} (v13);

\draw[-triangle 60] (v22) edge node[above] {$h_1$} (v11);

\draw[-triangle 60] (v35) edge node[above] {$h_2$} (v24);

\draw[-triangle 60] (v33) edge node[above] {$h_2$} (v22);

\draw[-triangle 60] (v13) edge node[above] {$\overline{h_1}$}  (v22);

\draw[-triangle 60] (v24) edge node[above] {$\overline{h_2}$}  (v33);

\end{tikzpicture}
\endpgfgraphicnamed
\caption{Vector space $\Rep^q(Q;v,w)$}
\label{fig:A_3_rep_space}
\end{figure}

First, assume $q$ is not a root of unity. Then the quotients
$\mathcal{M}^q(v,w)$ and $\mathcal{M}^q_0(v,w)$ do not depend on
$q$. They will be called the \emph{graded quiver varieties} and
denoted by $\grProjQuot(v,w)$,
$\grAffQuot(v,w)$ respectively. Let $\grAffQuot(w)$ denote the natural
union $\cup_v \grAffQuot(v,w)$. This is a finite dimensional affine
variety with a stratification into the
\emph{regular strata} $$\grAffQuot(w)=\sqcup_{v:w-C_qv\geq 0} \grRegStratum(v,w).$$

Similarly, assume $q$ equals $\epsilon$, which is a root of
unity. $\cycProjQuot(v,w)$ and $\cycAffQuot(v,w)$  will be called the
\emph{cyclic quiver varieties}. Let $\cycAffQuot(w)$ denote the
natural union $\cup_v \cycAffQuot(v,w)$. 

\begin{Prop}[{\cite[Section 2.5]{Nakajima01}}]\label{prop:finite_dominant}
Assume the quiver $Q$ is of Dynkin type $\mathbb{A}$, $\mathbb{D}$,
$\mathbb{E}$. Then the union $\cycAffQuot(w)$ is
  finite-dimensional with a stratification into
the regular strata
\begin{align}
  \cycAffQuot(w)=\sqcup_{v:w-C_qv\geq 0} \cycRegStratum(v,w).
\end{align}

\end{Prop}
% In particular, for any $w$, there exists finitely many
%   $v$, such that $(v,w)$ is $l$-dominant.

The properties of the cyclic quiver varieties are similar to those of the graded quiver varieties, expect for the following two
important differences:
\begin{itemize}
\item the linear map $C_q$ (q-Cartan matrix) is not injective;
\item it is not known if the smooth cyclic quiver variety $\cycProjQuot(v,w)$ is connected or not.
\end{itemize}

The smooth cyclic quiver variety $\cycProjQuot(v,w)$ is
pure-dimensional, \cf \cite[(4.1.6)]{Nakajima01}. For any $v$, choose a set $\{\alpha_v\}$ such that it parameterizes the connected
component of $\cycProjQuot(v,w)$. For any $l$-dominant pair $(v,w)$,
since the restriction of $\pi$ on the regular stratum
$\cycRegStratum(v,w)$ is a homeomorphism, the set $\set{\alpha_v}$ naturally parameterizes the connected components of this regular stratum:
\begin{align}
  \cycRegStratum(v,w)=\sqcup_{\alpha_{v}}\cycRegStratum^{;\alpha_{v}}(v,w).
\end{align}

Let $1_{\cycProjQuot(v,w)}$ denote the perverse sheaf associated with
the trivial local system of rank $1$ on $\cycProjQuot(v,w)$. Denote
the perverse sheaf $\pi_{!}(1_{\cycProjQuot(v,w)})$ by $\pi(v,w)$.

Using the transverse slice theorem (\cf \cite{Nakajima01}), we can
simplify the decomposition of
$\pi(v,w)$ as follows (\cf the proof of Theorem 8.6 in \cite{Nakajima04}):
\begin{align}
  \label{eq:decomposition_pi}
\pi(v,w)=\sum_{v':w-C_q v'\geq 0,\ v'\leq v}a_{v,v';w}(t)\cL(v',w),
\end{align}
where we denote $\cF[d]^{\oplus m}$ by $mt^d\cF$ for any sheaf $\cF$ and
$m\in \N$, $d\in Z$, and we define 
\begin{align}
  \cL(v',w)=IC(\cycRegStratum(v',w))=\oplus_{\alpha_{v'}}IC(\cycRegStratum^{;\alpha_{v'}}(v',w)).
\end{align}
Notice that we have $a_{v,v';w}(t)\in\N[t^\pm]$, $a_{v,v';w}(t^{-1})=a_{v,v';w}(t)$, and $a_{v,v;w}=1$. We
do not know if $\cL(v',w)$ is a simple perverse sheaf or not.

For any decomposition $w=w^1+w^2$, we have the restriction functor
between the derived category of constructible sheaves
\begin{align*}
  \tRes^w_{w^1,w^2}:\cD_c(\cycAffQuot(w))\ra
\cD_c(\cycAffQuot(w^1)\times \cycAffQuot(w^2)).
\end{align*}

By \cite{VaragnoloVasserot03}, $\tRes^w_{w^1,w^2}(\pi(v,w))$ equals  
\begin{align}\label{eq:restriction_pi}
\oplus_{v^1+v^2=v}\pi(v^1,w^1)\boxtimes  \pi(v^2,w^2)[d((v^2,w^2),(v^1,w^1))-d((v^1,w^1),(v^2,w^2))],
\end{align}
where the bilinear form $d(\ ,\ )$ is given by
\begin{align}
  \label{eq:dimension}
d((v^1,w^1),(v^2,w^2))=(w^1-C_qv^1)\cdot \sigma^* v^2+ v^1\cdot \sigma^*w^2.  
\end{align}

For each $w$, the Grothendieck group $K_w$ is defined as the
free abelian group generated by the perverse sheaves $\cL(v,w)$
appearing in \eqref{eq:decomposition_pi}. It has two
$\tBase$-bases: $\{\pi(v,w)|w-C_qv\geq 0,\cycRegStratum(v,w)\neq \emptyset\}$ and $\{\cL(v,w)\}$. Then its dual $R_w=\Hom_\tBase(K_w,\tBase)$ has the corresponding dual bases
$\{\chi(v,w)|w-C_qv\geq 0,\cycRegStratum(v,w)\neq \emptyset\}$ and $\{\can(v,w)\}$. Notice that, throughout this paper, we only define $\cL(v,w)$,
$\can(v,w)$ for the $l$-dominant pairs $(v,w)$ such that $\cycRegStratum(v,w)\neq \emptyset$.

The restriction functors induce an $\N^{\hat{I}}$-graded coassociative comultiplication on
the $\N^{\hat{I}}$-graded Grothendieck group $\oplus_{w}K_{w}$, which we
denote by $\tRes$.

\subsection{Quiver varieties and quiver representations}
\label{sec:quiver_rep}
Let $\Rep(Q)$ denote the category of
left $\C
Q$-modules. Let $\cD^b(Q)$ denote the bounded derived category of
$\Rep(Q)$ with the shift functor $\Sigma$. In $\cD^b(Q)$, we have Auslander-Reiten triangles. Also, let $\nu$
denote the derived tensor with the bimodule $\Hom_{\C
  Q}(\C Q,\C)$. Then we have
\begin{align*}
  D\Hom_{\cD^b(Q)}(x,y)=\Hom_{\cD^b(Q)}(y,\nu x), \forall x,y\in\cD^b( Q).
\end{align*}
By abuse of notation, we use $\tau$ to denote the Auslander-Reiten
translation which is defined as
$\Sigma^{-1}\nu$.

Let $\Ind\cD^b(Q)$ be a full subcategory of
$\cD^b(Q)$ whose objects form a set of representatives of the
isoclasses of the indecomposable objects of $\cD^b(Q)$ such that it is
stable under $\tau$ and $\Sigma$. Its subcategory $\Ind \Rep(Q)$
is naturally defined.

Assume $q$ is not a root of unity, then we can choose a natural
identification of $\sigma\hat{I}$ with (the objects of) $\Ind \cD^b(Q)$ such
that it commutes with $\tau$.

Define 
\begin{align}\label{eq:define_W}
  \begin{split}
    W^+&=\oplus_{x\in\Ind\Rep(Q)} \N e_{\sigma x},\\
    V^+&=\oplus_{x\in\Ind\Rep(Q),\ x\ \mathrm{is\ not\ injective}}
    \N
    e_{x},\\
    W^S&=\oplus_{x\in\{S_i,i\in I\}} \N e_{\sigma x}.
  \end{split}
\end{align}

\begin{Eg}[Quiver type $A_3$]
Let us continue Example \ref{eg:A_3_quiver}. The vertices in $\sigma
\hat{I}$ take the form $(i,q^{i+2d})$, $d\in\Z$, \cf Figure
\ref{fig:A_3_I}. On the other hand, the Auslander Reiten
quiver of $\Ind\cD^b(Q)$ is given in Figure \ref{fig:A_3_AR_quiver}. Notice that
the projective $\C Q$-module $P_1$ is also the simple module $S_1$.

So we can
identify $\sigma \hat{I}$ with the objects of $\Ind\cD^b(Q)$ by sending
the vertex $(i,q^{i+2d})$ to the object $\tau^{-d}P_i$. Then
Figure \ref{fig:A_3_rep_space} becomes Figure
\ref{fig:A_3_rep_space_module}. It follows that the dimension vectors in $W^+$ concentrate at the
vertices $\sigma x$, $x\in\Ind\C Q-mod$, those in $V^+$ concentrate at
$S_1$, $P_2$, $S_2$, and those in $W^S$ concentrate at $\sigma S_i$, $i=1,2,3$.
\end{Eg}

\begin{figure}[htb!]
 \centering
\beginpgfgraphicnamed{A_3_AR_quiver}
\begin{tikzpicture}[scale=0.35]
\node [] (v11) at (-8,-6) {$P_1$};
\node [] (v22) at (-4,-2) {$P_2$};
\node [] (v33) at (0,2) {$P_3$};
\node [] (v13) at (0,-6) {$S_2$};
\node [] (v24) at (4,-2) {$I_2$};
\node [] (v35) at (8,2) {$\Sigma P_1$};
\node [] (v15) at (8,-6) {$S_3$};
\node [] (v26) at (12,-2) {$\Sigma P_2$};
\node [] (v37) at (16,2) {$\Sigma S_2$};
\node [] (v17) at (16,-6) {$\Sigma P_3$};
\node [] (v28) at (20,-2) {$\Sigma I_2$};
\node [] (v39) at (24,2) {$\Sigma S_3$};

\node [] (v20) at (-8,-2) {$\cdots$};
\node [] (v31) at (24,-2) {$\cdots$};

\draw[-triangle 60] (v11) edge  (v22);
\draw[-triangle 60] (v22) edge  (v33);
\draw[-triangle 60] (v33) edge  (v24);
\draw[-triangle 60] (v22) edge  (v13);

\draw[-triangle 60] (v13) edge  (v24);
\draw[-triangle 60] (v24) edge  (v35);
\draw[-triangle 60] (v35) edge  (v26);
\draw[-triangle 60] (v24) edge  (v15);

\draw[-triangle 60] (v15) edge  (v26);
\draw[-triangle 60] (v26) edge  (v37);
\draw[-triangle 60] (v37) edge  (v28);
\draw[-triangle 60] (v26) edge  (v17);

\draw[-triangle 60] (v17) edge  (v28);
\draw[-triangle 60] (v28) edge  (v39);

\draw[-latex][dashed] (v13) edge  node[below] {$\tau$} (v11);
\draw[-latex][dashed] (v15) edge  node[below] {$\tau$} (v13);
\draw[-latex][dashed] (v17) edge  node[below] {$\tau$} (v15);
\draw[-latex][dashed] (v24) edge  (v22);
\draw[-latex][dashed] (v26) edge  (v24);
\draw[-latex][dashed] (v28) edge  (v26);
\draw[-latex][dashed] (v35) edge  (v33);
\draw[-latex][dashed] (v37) edge  (v35);
\draw[-latex][dashed] (v39) edge  (v37);

\end{tikzpicture}
\endpgfgraphicnamed
\caption{Auslander-Reiten quiver of $\Ind\cD^b(Q)$}
\label{fig:A_3_AR_quiver}
\end{figure}
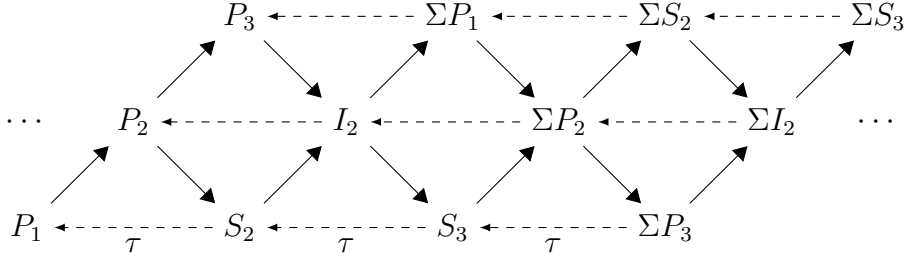

\begin{figure}[htb!]
 \centering
\beginpgfgraphicnamed{A_3_rep_space_module}
\begin{tikzpicture}[scale=0.55]

\node [color=blue] (w10) at (-12,-6) {$W(\sigma S_1)$};
\node [color=blue] (w21) at (-8,-2) {$W(\sigma P_2)$};
\node [color=blue] (w32) at (-4,2) {$W(\sigma P_3)$};
\node [color=blue] (w12) at (-4,-6) {$W(\sigma S_2)$};
\node [color=blue] (w23) at (0,-2) {$W(\sigma I_2)$};
\node [color=blue] (w34) at (4,2) {$W(\sigma \Sigma S_1)$};
\node [color=blue] (w14) at (4,-6) {$W(\sigma S_3)$};

% \node [color=blue] (w14) at (4,-6) {$W(1,q^4)$};
% \node [color=blue] (w25) at (8,-2) {$W(2,q^5)$};
% \node [color=blue] (w36) at (12,2) {$W(3,q^6)$};
\node [color=blue] (w15) at (8,-2) {$\cdots$};
\node [color=blue] (w25) at (8,-6) {$\cdots$};

\node [] (v11) at (-8,-6) {$V(S_1)$};
\node [] (v22) at (-4,-2) {$V(P_2)$};
\node [] (v33) at (0,2) {$V(P_3)$};
\node [] (v13) at (0,-6) {$V(S_2)$};
\node [] (v24) at (4,-2) {$V(I_2)$};
\node [] (v35) at (8,2) {$V(\Sigma S_1)$};

%\node [color=blue] (w30) at (-12,2) {$W_3(1)$};
%\node [color=blue] (w16) at (12,-6) {$W_1(q^6)$};
%\node [] (v20) at (-12,-2) {$V_2(1)$};
%\node [] (v31) at (-8,2) {$V_3(q)$};
%\node [] (v15) at (8,-6) {$V_1(q^5)$};
%\node [] (v26) at (12,-2) {$V_2(q^6)$};
\node [] (v20) at (-12,-2) {$\cdots$};
\node [] (v31) at (-8,2) {$\cdots$};
% \node [] (v15) at (8,-6) {$\cdots$};
% \node [] (v26) at (12,-2) {$\cdots$};

% \draw[-triangle 60][color=red] (w14) edge  node[above] {$\alpha_1$} (v13);
\draw[-triangle 60][color=red] (w12) edge  node[above] {$\alpha_1$} (v11);

% \draw[-triangle 60][color=red] (w25) edge  node[above] {$\alpha_2$} (v24);
\draw[-triangle 60][color=red] (w23) edge  node[above] {$\alpha_2$} (v22);

% \draw[-triangle 60][color=red] (w36) edge  node[above] {$\alpha_3$} (v35);
\draw[-triangle 60][color=red] (w34) edge  node[above] {$\alpha_3$} (v33);

\draw[-triangle 60][color=red] (w14) edge  node[above] {$\alpha_1$} (v13);

\draw[-triangle 60][color=blue] (v13) edge node[above] {$\beta_1$} (w12);
\draw[-triangle 60][color=blue] (v11) edge node[above] {$\beta_1$} (w10);

\draw[-triangle 60][color=blue] (v24) edge node[above] {$\beta_2$} (w23);
\draw[-triangle 60][color=blue] (v22) edge node[above] {$\beta_2$} (w21);

\draw[-triangle 60][color=blue] (v35) edge node[above] {$\beta_3$} (w34);
\draw[-triangle 60][color=blue] (v33) edge node[above] {$\beta_3$} (w32);

\draw[-triangle 60] (v24) edge node[above] {$h_1$} (v13);

\draw[-triangle 60] (v22) edge node[above] {$h_1$} (v11);

\draw[-triangle 60] (v35) edge node[above] {$h_2$} (v24);

\draw[-triangle 60] (v33) edge node[above] {$h_2$} (v22);

\draw[-triangle 60] (v13) edge node[above] {$\overline{h_1}$}  (v22);

\draw[-triangle 60] (v24) edge node[above] {$\overline{h_2}$}  (v33);

\end{tikzpicture}
\endpgfgraphicnamed
\caption{Vector space $\Rep^q(Q;v,w)$}
\label{fig:A_3_rep_space_module}
\end{figure}
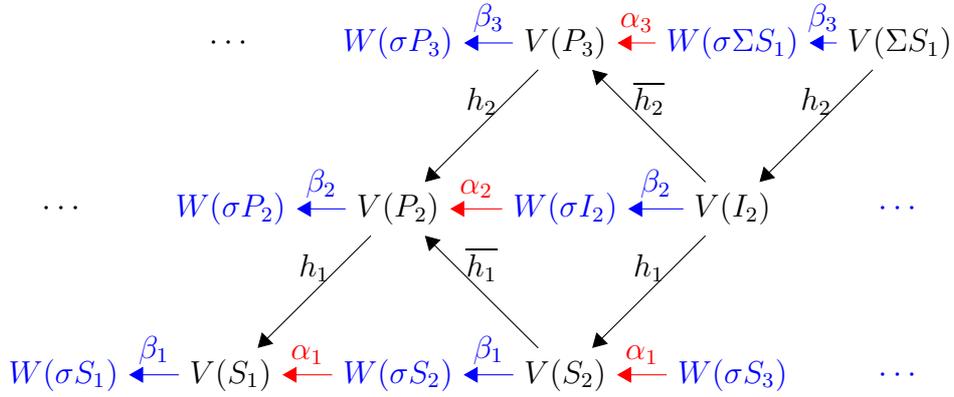

Recall that a pair $(v,w)$ is called \emph{$l$-dominant} if
$w-C_qv\geq 0.$ The vector spaces $W^+$, $V^+$, $W^S$ are defined in \eqref{eq:define_W}. We shall use the following combinatorial property of the $l$-dominant
pairs.

\begin{Thm}[\cite{LP_rep_alg}]\label{thm:abelian_representative}
Assume $q$ is not a root of unity. Then for any $\tw\in W^+$, there exists a unique $l$-dominant pair
$(v,w)\in V^+\times W^S$, such that $w-C_qv=\tw$.
\end{Thm}

\section{Grothendieck rings arising from cyclic quiver varieties}

\subsection{Constructions and main results}
\label{sec:construction}

We follow the conventions in section \ref{sec:preliminaries}. Let $h$
denote the Coxeter number of the Dynkin type of the quiver $Q$. We
make the following convention from now on:
\begin{center}
  Fix $\epsilon$ to be a $2h$-th primitive root
  of unity and, by default, take $q$ to be $\epsilon$.
\end{center}

It follows that the automorphism $\tau^{h}$ on $I\times \langle
\epsilon\rangle$ equals $1$. The subset
$\sigma\hat{I}$ of $I\times \langle
\epsilon\rangle$ has the cardinality $nh$.

Let us choose the natural covering map $\pi$ from the set $\Ind
\cD^b(Q)$ to $\sigma\hat{I}$ which sends the $i$-th projective $P_i$
to $(i,\epsilon\xi_i)$ and commutes with $\tau$, namely,
$\tau\pi(M)=\pi(\tau M)$.

Choose a section $M_?$ of this
covering map $\pi$, such that $\pi M_x=x$ for any $x\in
\sigma\hat{I}$. We further require that the image of $\sigma\hat{I}$
under $M_?$ is contained in $(\Ind \Rep(Q))\sqcup
(\Sigma(\Ind\Rep(Q)))$. When the context is clear, we simply denote
$M_x$ by $x$ and omit the notation of the covering map $\pi$.

Notice that the image of the section map $M_?$ is not closed
under $\tau$ nor $\Sigma$. 

\begin{Eg}
In Example \ref{eg:A_3_quiver}, we can take $q=\epsilon$ to be a
primitive $8$-th root of the unity. Then the vertices $\sigma \hat{I}$
take the form $(i,q^{i+2d})$, $i\in I$, $d\in\{0,1,2,3\}$. We can
construct the section
map from $\sigma\hat{I}$ to $\Ind\cD^b(Q)$ which sends $(i,q^{i+2d})$
to $\tau^{-d}P_i$ (these are the objects already drawn in Figure \ref{fig:A_3_AR_quiver}).
\end{Eg}

The shift functor $\Sigma$ induces an automorphism $\Sigma$ on the set $\Ind
\cD^b(Q)$. It is inherited by $\sigma\hat{I}$. We extend
this automorphism $\Sigma$ to $I\times \langle \epsilon\rangle$ by
requiring $\Sigma \sigma=\sigma\Sigma$. It follows that $\Sigma^2=1$.

Let $W^+$, $V^+$,
$W^S$ be defined as in \eqref{eq:define_W}. We also define
\begin{align}
  \begin{split}
    W^-=\Sigma^* W^+,\\
    V^-=\Sigma^* V^-,\\
    W^{\Sigma S}=\Sigma^* W^S.
  \end{split}
\end{align}

For any $i\in I$, we define 
\begin{align}\label{eq:define_vf}
  \begin{split}
  w^{f_i}&=e_{\sigma S_i}+e_{\sigma\Sigma S_i},\\
W^0&=\oplus_{i\in I}\N w^{f_i},\\
v^{f_i}&=\sum_{x\in \sigma\hat{I}}\dim\Hom_{\cD^b(Q)}(S_i,M_x)
e_x,\\
v^{\Sigma f_i}&=\Sigma^* v^{f_i},\\
V^0&=\oplus_{i\in I} (\N v^{f_i}\oplus \N v^{\Sigma f_i}).
\end{split}
\end{align}

Following section \ref{sec:quiver_variety}, we consider the Grothendieck group $K=\oplus_{w\in W^S\oplus W^{\Sigma S}} K_w$. Its
$\N^{\hat{I}}$-graded dual $R=\oplus_{w\in W^S\oplus W^{\Sigma S}}R_w$ has the
multiplication $\tOtimes$ induced by the comultiplication $\tRes$ of
$K$. 

It follows from \cite{VaragnoloVasserot03} that we have
\begin{align}
  \label{eq:multiplication_can}
  \can(v^1,w^1)\tOtimes\can(v^2,w^2)=\sum_v c^v_{v^1,v^2}(t)\can(v,w^1+w^2),
\end{align}
such that $c^v_{v^1,v^2}(t)\in\N[t^\pm]$, $c^v_{v^1,v^2}=0$ whenever
$v<v^1+v^2$, and
$c^{v^1+v^2}_{v^1,v^2}=t^{d((v^2,w^2),(v^1,w^1))-d((v^1,w^1),(v^2,w^2))}$. The
term $c^{v_1+v_2}_{v^1,v^2}(t)\can(v_1+v_2,w^1+w^2)$ is called the leading term of $\RHS$ of \eqref{eq:multiplication_can}.

\begin{Prop}[{\cite[Theorem 7.3]{HernandezLeclerc11}}]\label{prop:generator}
(1) $R^+=\oplus_{w\in W^S}R_w$ is the $\Z[t^\pm]$-algebra generated by $\{\can(0,e_{\sigma
  S_i}),\ i\in I\}$ with respect to the product $\tOtimes$.

(2) $R^-=\oplus_{w\in W^{\Sigma
    S}}R_w$ is the $\Z[t^\pm]$-algebra generated by $\{\can(0,e_{\sigma \Sigma S_i}),\ i\in
I\}$ with respect to the product $\tOtimes$.
\end{Prop}
We also define $R^0$ to be the algebra generated by
\begin{align}
  \{\can(v^{f_i},w^{f_i}),\can(v^{\Sigma f_i},w^{f_i}),i\in I\}.
\end{align}
We will call $\can(0,e_{\sigma
  S_i})$, $\can(v^{f_i},w^{f_i})$, $\can(v^{\Sigma f_i},w^{f_i})$, $\can(0,e_{\sigma\Sigma
  S_i})$, $i\in I$, the \emph{Chevalley generators} of the
Grothendieck ring $R$. 

\begin{Rem}
  The generators $\can(0,e_{\sigma S_i})$, $\can(0,e_{\sigma\Sigma
    S_i})$ should be compared with the generators $y_{i,0}$, $y_{i,1}$
  in \cite[Theorem 7.3]{HernandezLeclerc11} for derived categories
  respectively. We shall show that the relation (R1) in
  \cite[Theorem 7.3]{HernandezLeclerc11} holds for our generators. But
  the relation (R2) does not hold in our case. See Example
  \ref{eg:compare_product} for more details.
\end{Rem}

Let us use $(\ )_{t^\Hf}$ and $(\ )_{\Q(t^\Hf)}$ to denote the
extensions $(\ )\otimes \Z(t^\Hf)$ and $(\ )\otimes
\Q(t^\Hf)$ respectively.
% ,\\
%  R^+_{\Q(t^\Hf)}&= R^+\otimes_\tBase \Q(t^\Hf),\\
%  R^0_{\Q(t^\Hf)}&=R^0\otimes_\tBase \Q(t^\Hf),\\
%   R^-_{\Q(t^\Hf)}&=  R^-\otimes_\tBase \Q(t^\Hf),\\
%  \tUt(\mfr{g})_{\Q(t^\Hf)}&= \tUt(\mfr{g})\otimes_{Q(t)} \Q(t^\Hf),\\
%   \tUt(\mfr{n}^+)_{\Q(t^\Hf)}&= \tUt(\mfr{n}^+)\otimes_{Q(t)} \Q(t^\Hf),\\
% \tUt(\mfr{h})_{\Q(t^\Hf)}&= \tUt(\mfr{h})\otimes_{Q(t)} \Q(t^\Hf),\\
%   \tUt(\mfr{n}^-)_{\Q(t^\Hf)}&=\tUt(\mfr{n}^-)\otimes_{Q(t)} \Q(t^\Hf).
% \end{align*}

% We can naturally view $\hat{I}$ as a
% basis of $K_0(\Rep(Q))\oplus K_0(\Sigma(\Rep(Q)))$ such that
% $\sigma x$ is identified with the object $x$.
Let $\Phi$ be the linear
map from $\N^{\hat{I}}$ to the Grothendieck group $K_0(\Rep(Q))\oplus K_0(\Sigma(\Rep(Q)))$ such
that $\Phi(e_{\sigma x})=x$. For any elements $x=(x^1,x^2)$, $y=(y^1,y^2)\in K_0(\Rep(Q))\oplus
  K_0((\Sigma\Rep(Q)))$, define the following bilinear forms as
  combinations of the Euler forms
  \begin{align}\label{eq:Euler_twist}
    \langle x,y\rangle _a=\langle x^1,y^1\rangle-\langle y^1,x^1\rangle+\langle
    x^2,y^2\rangle -\langle y^2,x^2\rangle,\\
(x,y)=\langle x^1,y^1\rangle+\langle y^1,x^1\rangle+\langle
    x^2,y^2\rangle +\langle y^2,x^2\rangle.
  \end{align}
 
 Following the convention in Section \ref{sec:quiver_variety}, for
  any $w=w^1+w^2$, $\tRes^{w}_{w^1,w^2}$ is a homomorphism from
  $(K_w)_{t^\Hf}$ to  $(K_w)_{t^\Hf}\otimes_{\Z[t^{\pm\Hf}]} (K_w)_{t^\Hf}$.
  Define\footnote{This choice of the degree arises from the comparison
    of the equations in Proposition \ref{prop:EK} with the defining relations of
    $\tUt(\mfr{g})$. It is also an anti-symmetrized version of the
    twist used by \cite{Bridgeland:Hall}.} its deformation $\res^{w}_{w^1,w^2}$ to be
  $\tRes^w_{w^1,w^2}t^{-\Hf\langle \Phi(w^1),\Phi(w^2) \rangle_a}$. Then we
  obtain a (coassociative) comultiplication $\res$ on $K_{t^\Hf}$ and correspondingly a
  multiplication $\otimes$ on $R_{t^\Hf}$. We compare it with the
  twisted products in \cite{Hernandez02}\cite{HernandezLeclerc11}\cite{Nakajima01} in Example \ref{eg:compare_product}.

For any $w\in W^+$, $\Phi(w)$ can be viewed as a $\C
Q$-module. We define $\deg\Phi(w)$ to be the total dimension of
$\Phi(w)$, and the bilinear form
\begin{align}
  N(\Phi(w))&=( \Phi(w),\Phi(w))-\deg\Phi(w).
\end{align}

Let
$B^*_K=\set{B^*_K(w)|w\in W^+}$ denote the \emph{dual canonical
basis} of $U_t(\mfr{n}^+)$ with respect to Kashiwara's linear
form $(\ ,\ )_K$. Define the rescaled dual
canonical basis $\tB^*_K$ to be $\set{\tB^*_K(w)|w\in W^+}$ such that $\tB^*_K(w)=t^{\Hf N(\Phi(w))}B^*_K(w)$.

The following was the main result of \cite{HernandezLeclerc11} for
graded quiver varieties with a generic choice of $q$.
\begin{Thm}[{One-half quantum group, \cite[Theorem 6.1]{HernandezLeclerc11}}] \label{thm:isom_coordinate_ring}
(1)  There exists an algebra isomorphism $\tKappa$ from the Grothendieck ring $(R^+_{\Q(t^\Hf)},\otimes)$ to the
   one-half quantum group $\Ut(\mfr{n}^+)_{\Q(t^\Hf)}$, such that
   $$\tKappa\can(0,e_{\sigma S_i})=E_i,\ \forall i\in I$$.

(2) This isomorphism identifies the basis $\{\can(v,w),w\in W^S\}$ with
   the rescaled dual canonical basis $\tB^*_K$ such that $\tKappa(\can(v,w))=\tB_K^*(w-C_qv)$.
\end{Thm}
% The second assertion is essentially based on the $T$-system studied in \cite{GeissLeclercSchroeer11}.

\begin{Thm}[Triangular decomposition]\label{thm:triangular_decomposition}
  The ring $(R,\otimes)$ (\resp $(R,\tOtimes)$) decomposes into the tensor product of its subalgebras:
  \begin{align*}
    R=R^+\otimes_\tBase R^0\otimes_\tBase R^-.
  \end{align*}
\end{Thm}

% We call $(R_{t^\Hf},\otimes)$ the
  % \emph{twisted Grothendieck ring}.

\begin{Thm}\label{thm:isom}
(i) There exists an algebra isomorphism $\kappa$ from $(R_{
  \Q(t^\Hf)},\otimes)$ to $\tUt(\mfr{g})_{\Q(t^\Hf)}$, such that we have
$\tUt(\mfr{n}^+)_{\Q(t^\Hf)}= \kappa R^+_{\Q(t^\Hf)}$,
$ \tUt(\mfr{h})_{\Q(t^\Hf)}=\kappa R^0_{\Q(t^\Hf)}$, $ \tUt(\mfr{n}^-)_{\Q(t^\Hf)}=\kappa R^-_{\Q(t^\Hf)}$, and for any $i\in I$,
\begin{align*}
\kappa(t^\Hf)&=t^\Hf,\\
\kappa\can(0,e_{\sigma S_i})&=\frac{1-t^2}{t}E_i,\\
\kappa\can(v^{\Sigma f_i},w^{f_i})&=K_i,\\
\kappa\can(v^{f_i},w^{f_i})&=K_i',\\
 \kappa \can(0,e_{\sigma \Sigma S_i})&= \frac{t^2-1}{t}F_i.
\end{align*}
(ii)
  Let $I$ be the ideal of $(R,\otimes)$ generated by the center
  elements $\can(v^{f_i}+v^{\Sigma f_i},2 w^{f_i})-1$, $i\in I$. Then the map $\kappa$
  induces an isomorphism between the quotient ring $R/I$ and the
  quantum group $\Ut(\mfr{g})$.
\end{Thm}

Finally, we consider the dual canonical basis of $\Ut(\mfr{n}^+)$ with respect to Lusztig's bilinear
form $(\ ,\ )_L$, which is denoted by
$$B^*_L=\set{B^*_L(w)|w\in W^+}.$$

Define the rescaled dual
canonical basis to be $\tB^*_L=\set{\tB^*_L(w)|w\in W^+}$ such that
\begin{align}
   \tB^*_L(w)=t^{\Hf N(\Phi(w))-\deg\Phi(w)}B^*_L(w).
\end{align}

It is not obvious to see that on the Grothendieck ring $R^+$, our
twisted product $\otimes$ agrees with the non-commutative
multiplication $*$ defined in \cite{HernandezLeclerc11}, which we will
show in the last section. Once we see that they coincide on the
subalgebra $R^+$, \cite{HernandezLeclerc11}[Theorem 6.1] is translated
as the following.
\begin{Thm}[{\cite{HernandezLeclerc11}[Theorem 6.1]}]\label{thm:Lusztig_basis}
  The isomorphism $\kappa$ identifies $\set{\can(v, w),w\in W^S}$ with
  $\tB^*_L$ such that $\kappa \can(v,w)=\tB^*_L(w-C_qv)$.
\end{Thm}

\subsection{Examples}
\label{sec:example}

\begin{Eg}[Type $\mfr{sl}_2$]\label{eg:SL_2}
  Assume that the quiver $Q$ consists of a single point. Then $h$ equals $2$. $\epsilon$
  is a $4$-th primitive root of unity. The vector space
  $\Rep^\epsilon(Q;v,w)$ for cyclic quiver varieties is given by Figure
  \ref{fig:A_1_rep_space}.

The Chevalley generators of the
  Grothendieck ring $(R,\otimes)$ are given by
  \begin{align*}
    \can(0,e_{\sigma S}),\ \can(e_S,e_{\sigma S+\sigma \Sigma S}),\
    \can(e_{\Sigma S},e_{\sigma S+\sigma \Sigma S}),\ \can(0, e_{\sigma\Sigma S}).
  \end{align*}
Theorem \ref{thm:isom} identifies $R_{\Q(t^\Hf)}$ with
$\tUt(\mfr{sl}_2)_{\Q(t^\Hf)}$ and the above generators with $E$, $K'$,
$K$, $F$ respectively.
\end{Eg}

\begin{figure}[htb!]
\centering
\beginpgfgraphicnamed{A_1_rep_space}
\begin{tikzpicture}[scale=0.5]
\node [color=blue] (w10) at (-12,-6) {$W(\sigma S)$};
\node [color=blue] (w12) at (-4,-6) {$W(\sigma\Sigma S)$};
\node [color=blue] (w14) at (4,-6) {$W(\sigma  S)$};

\node [] (v11) at (-8,-6) {$V(S)$};
\node [] (v13) at (0,-6) {$V(\Sigma S)$};

\node [color=blue] (deg0) at (-12,-4) {$\mathrm{height}=1$};
\node [color=blue] (deg1) at (-8,-4) {$\epsilon$};
\node [color=blue] (deg2) at (-4,-4) {$\epsilon^2$};
\node [color=blue] (deg3) at (0,-4) {$\epsilon^3$};
\node [color=blue] (deg4) at (4,-4) {$\epsilon^4=1$};

\draw[-triangle 60][color=red] (w14) edge  node[above] {$\alpha$} (v13);
\draw[-triangle 60][color=red] (w12) edge  node[above] {$\alpha$} (v11);

\draw[-triangle 60][color=blue] (v13) edge node[above] {$\beta$} (w12);
\draw[-triangle 60][color=blue] (v11) edge node[above] {$\beta$} (w10);

\end{tikzpicture}
\endpgfgraphicnamed
\caption{$\Rep^\epsilon(Q;v,w)$ for cyclic quiver varieties of type $\mfr{sl}_2$.}
\label{fig:A_1_rep_space}
\end{figure}

\begin{Eg}[Type $\mfr{sl}_3$]\label{eg:SL_3}
  Assume that the quiver $Q$ takes the form $(2\xra{h} 1)$. Then $h$ equals
  $3$. $\epsilon$ is a $6$-th root of unity. The vector space
  $\Rep^\epsilon(Q;v,w)$ for cyclic quiver varieties is given by Figure
  \ref{fig:A_2_rep_space}.

The Chevalley generators of the
  Grothendieck ring $(R,\otimes)$ are given by
  \begin{align*}
    \can(0,e_{\sigma S_i}),\ i=1,2\\
\can(e_{S_1}+e_{P_2},e_{\sigma S_1+\sigma \Sigma S_1}),\ 
\can(e_{S_2}+e_{\Sigma S_1},e_{\sigma S_2+\sigma \Sigma S_2}),\\
\can(e_{\Sigma S_1}+e_{\Sigma P_2},e_{\sigma S_1+\sigma \Sigma S_1}),\ 
\can(e_{\Sigma S_2}+e_{S_1},e_{\sigma S_2+\sigma \Sigma S_2}),\\
  \can(0, e_{\sigma\Sigma S_i}),\ i=1,2.
  \end{align*}
Theorem \ref{thm:isom} identifies $R_{\Q(t^\Hf)}$ with $\tUt(\mfr{sl}_3)_{\Q(t^
\Hf)}$ and the above generators with $E_i$, $K_1'$, $K_2'$, $K_1$,
$K_2$, $F_i$ respectively.
\end{Eg}

\begin{figure}[htb!]
 \centering
\beginpgfgraphicnamed{A_2_rep_space}
\begin{tikzpicture}[scale=0.55]

\node [color=blue] (w10) at (-12,-6) {$W(\sigma S_1)$};
\node [color=blue] (w21) at (-8,-2) {$W(\sigma P_2)$};
\node [color=blue] (w12) at (-4,-6) {$W(\sigma S_2)$};
\node [color=blue] (w23) at (0,-2) {$W(\sigma \Sigma S_1)$};
\node [color=blue] (w14) at (4,-6) {$W(\sigma \Sigma P_2)$};
\node [color=blue] (w25) at (8,-2) {$W(\sigma \Sigma S_2)$};

\node [] (v11) at (-8,-6) {$V(S_1)$};
\node [] (v22) at (-4,-2) {$V(P_2)$};
\node [] (v13) at (0,-6) {$V(S_2)$};
\node [] (v24) at (4,-2) {$V(\Sigma S_1)$};

\node [color=blue] (w16) at (12,-6) {$W(\sigma S_1)$};
\node [] (v20) at (-12,-2) {$V(\Sigma S_2)$};
\node [] (v15) at (8,-6) {$V(\Sigma P_2)$};
\node [] (v26) at (12,-2) {$V(\Sigma S_2)$};

\node [color=blue] (deg0) at (-12,0) {$\mathrm{height}=1$};
\node [color=blue] (deg1) at (-8,0) {$\epsilon$};
\node [color=blue] (deg2) at (-4,0) {$\epsilon^2$};
\node [color=blue] (deg3) at (0,0) {$\epsilon^3$};
\node [color=blue] (deg4) at (4,0) {$\epsilon^4$};
\node [color=blue] (deg5) at (8,0) {$\epsilon^5$} ;
\node [color=blue] (deg6) at (12,0) {$\epsilon^6=1$} ;

\draw[-triangle 60][color=red] (w16) edge  node[above] {$\alpha_1$} (v15);
\draw[-triangle 60][color=red] (w14) edge  node[above] {$\alpha_1$} (v13);
\draw[-triangle 60][color=red] (w12) edge  node[above] {$\alpha_1$} (v11);

\draw[-triangle 60][color=red] (w25) edge  node[above] {$\alpha_2$} (v24);
\draw[-triangle 60][color=red] (w23) edge  node[above] {$\alpha_2$} (v22);
\draw[-triangle 60][color=red] (w21) edge  node[above] {$\alpha_2$} (v20);

\draw[-triangle 60][color=blue] (v15) edge node[above] {$\beta_1$} (w14);
\draw[-triangle 60][color=blue] (v13) edge node[above] {$\beta_1$} (w12);
\draw[-triangle 60][color=blue] (v11) edge node[above] {$\beta_1$} (w10);

\draw[-triangle 60][color=blue] (v26) edge node[above] {$\beta_2$} (w25);
\draw[-triangle 60][color=blue] (v24) edge node[above] {$\beta_2$} (w23);
\draw[-triangle 60][color=blue] (v22) edge node[above] {$\beta_2$} (w21);

\draw[-triangle 60] (v26) edge node[above] {$h$} (v15);

\draw[-triangle 60] (v24) edge node[above] {$h$} (v13);

\draw[-triangle 60] (v22) edge node[above] {$h$} (v11);

\draw[-triangle 60] (v15) edge node[above] {$\overline{h}$}  (v24);
\draw[-triangle 60] (v13) edge node[above] {$\overline{h}$}  (v22);
\draw[-triangle 60] (v11) edge node[above] {$\overline{h}$}  (v20);
\end{tikzpicture}
\endpgfgraphicnamed
\caption{$\Rep^\epsilon(Q;v,w)$ for cyclic quiver varieties of type $\mfr{sl}_3$.}
\label{fig:A_2_rep_space}
\end{figure}
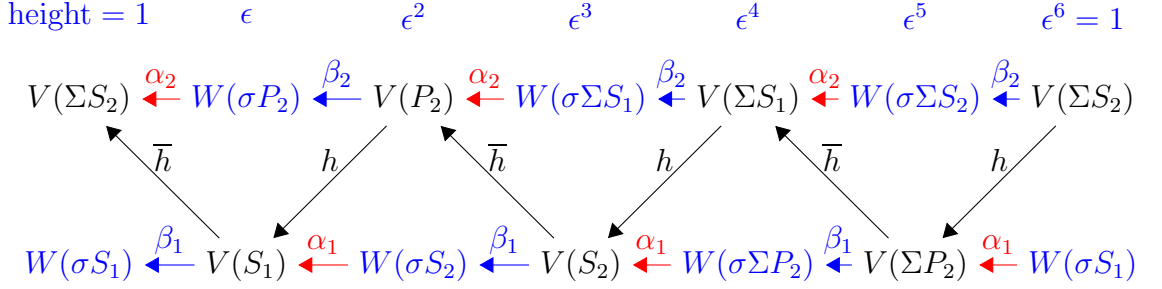

\begin{Eg}
\label{eg:compare_product}
We continue Example
\ref{eg:SL_3} and compare various twisted products.

Let us take Chevalley generators from the positive part and the negative part of the quantum group respectively. Our
twisted product $\otimes$ satisfies
\begin{align*}
  \can(0,e_{\sigma S_1})\otimes \can(0,e_{\sigma \Sigma S_2})=
  \can(0,e_{\sigma \Sigma S_2})\otimes \can(0,e_{\sigma S_1}).
\end{align*}
On the other hand, the twisted product in \cite[Theorem
7.3(R2)]{HernandezLeclerc11}\cite{Hernandez02}, if defined over the
pairs $(v,w)$, would demand the following relation
\begin{align*}
  \can(0,e_{\sigma S_1})\otimes \can(0,e_{\sigma \Sigma S_2})=
  t^{-(\alpha_1,\alpha_2)}\can(0,e_{\sigma \Sigma S_2})\otimes \can(0,e_{\sigma S_1}).
\end{align*}
Therefore, $\otimes$ is not the same as the product used in
\cite{HernandezLeclerc11}\cite{Hernandez02}.

$\otimes$ is not the twisted product in \cite{Nakajima04}
either. Recall that
our geometrical restriction functor is twisted by the Euler form in
\eqref{eq:Euler_twist}. But \cite{Nakajima04} twisted the geometrical
restriction functor by a different bilinear form $d_W$ for generic
$q$, which takes different values. For $q$ a root of unity, \cite{Nakajima04} (cf. also \cite[Theorem 3.5]{Hernandez04}) used the different twisted product $\tilde{\otimes}$ associated with the anti-symmetrized version of the bilinear form $d(\ ,\ )$, which gives us
\begin{align*}
d((0,e_{\sigma S_2}),(0,e_{\sigma S_1}))-d((0,e_{\sigma S_1}),(0,e_{\sigma S_2}))=0,\\
  \can(0,e_{\sigma S_1})\tilde{\otimes} \can(0,e_{\sigma S_2})=t^0\can(0,e_{\sigma  S_2}+e_{\sigma S_1})+\mathrm{other\ terms}.
\end{align*}
On the other hand, our product $\otimes$ would twist the above leading $t$-power by $t^{-\frac{1}{2}\langle S_1,S_2\rangle}$.

In fact, our twisted product $\otimes$ is defined for the pair $(v,w)$, where $v\in
\N^{\sigma \hat{I}}$, $w\in \N^{\hat{I}}$, while the the twisted products in \cite{Nakajima04}(for
generic $q$) and
\cite{HernandezLeclerc11}\cite{Hernandez02} are % used for studying
% quantum affine algebras' representations and 
defined over the
dimension vectors $w$. In order to compare our $\otimes$ with the
latter two products, we have to reduce
the pair $(v,w)$ to the dimension $w-C_q v$. This would demand the Cartan elements $\can(e_{S_1}+e_{P_2},e_{\sigma
  S_1+\sigma \Sigma S_1})$, $\can(e_{S_2}+e_{\Sigma S_1},e_{\sigma
  S_2+\sigma \Sigma S_2})$, $
\can(e_{\Sigma S_1}+e_{\Sigma P_2},e_{\sigma S_1+\sigma \Sigma S_1})$, $ 
\can(e_{\Sigma S_2}+e_{S_1},e_{\sigma S_2+\sigma \Sigma S_2})$ to be
center with respect to $\otimes$, which is not true. The author
does not know any non-trivial twisted
product defined over cyclic quiver varieties such that these Cartan
elements become center elements. The incompatibility of
our product $\otimes$ and the twisted products in \cite{Nakajima04}(for
generic $q$) and
\cite{HernandezLeclerc11}\cite{Hernandez02} could be expected because
abelian categories of $2$-periodic complexes are not subcategory of derived categories. % More details of
% root categories can be found in \cite{PengXiao97}.

Nevertheless, the restriction of the twisted product $\otimes$ on $\can(0,e_{\sigma S_i})$, $\can(0,e_{\sigma S_j})$, $i,j=1,2$,
agrees with that of \cite[Theorem 7.3(R1)]{HernandezLeclerc11} for
the corresponding elements, \cf Section \ref{sec:comparison}.
\end{Eg}
\section{Proofs}

For simplicity, we shall often denote $\Hom_{\cD^b(Q)}(\ ,\ )$ by
$\Hom(\ ,\ )$.

\subsection{l-dominant pairs}

\begin{Lem}\label{lem:wrap}
  For any $x\in \Ind\Rep(Q)$, $y\in \sigma\hat{I}$, we have
  \begin{align*}
    \Hom_{\cD^b(Q)}(x,\tau M_y)=\Hom_{\cD^b(Q)}(x,M_{\tau y}).
  \end{align*}
\end{Lem}
\begin{proof}
Notice that $\tau^h=1$ and $\sigma\hat{I}$ is
identified with $\Ind \Rep(Q)\sqcup \Sigma\Ind \Rep(Q)$. The statement obviously hold if $M_y$ is not a projective $\C
  Q$-module. On the other hand, assume $M_y$ is a projective $\C
  Q$-module. Then $\Sigma^{-1}M_{\tau y}$ is an injective $\C Q$-module and both hand sides vanish.
\end{proof}

\begin{Lem}
 For any $i\in I$, $w^{f_i}-C_q v^{f_i}$ vanishes.
\end{Lem}
\begin{proof}
For any $x\in\sigma\hat{I}$, we have an almost split triangle in $\cD^b(Q)$
\begin{align*}
  \tau M_x\ra E\ra M_x\ra \Sigma \tau M_x.
\end{align*}

For simplicity, we denote $\Hom_{\cD^b(Q)}(\ ,\ )$ by $\Hom(\ ,\ )$.
Applying the functor $\Hom(S_i,\ )$ to this triangle, we get a long
exact sequence
\begin{align*}
  \Hom(S_i,\Sigma^{-1}M_x)&\xra{\omega^1}\Hom(S_i,\tau M_x)\\
&\ra\Hom(S_i, E)\ra
  \Hom(S_i,  M_x)\xra{\omega^2}\Hom(S_i,\Sigma\tau M_x).
\end{align*}

By Lemma \ref{lem:wrap} and \eqref{eq:define_vf}, the
coordinate of $C_q v^{f_i}$ at the vertex $\sigma x$ is $$(C_qv^{f_i})_{\sigma x}=  \dim\Hom(S_i,\tau
M_x)-\dim\Hom(S_i,E)+\dim\Hom(S_i,M).$$

(i) Assume $x\in\Ind \Rep(Q)$. Then $\Hom(S_i,\Sigma^{-1}M_x)$ vanishes.

If $x\neq S_i$, we get $\omega^2=0$ by the universal property of
Auslander-Reiten triangles and consequently
\begin{align*}
(C_qv^{f_i})_{\sigma x}= \dim\Hom(S_i,\tau M_x)-\dim\Hom(S_i,E)+\dim\Hom(S_i,M)=0.
\end{align*}

For $x=S_i$, we get $\Ker w^2=0$ and 
\begin{align*}
  (C_qv^{f_i})_{\sigma S_i}=\dim\Hom(S_i,S_i)=1.
\end{align*}

(ii) Assume $x\in \Sigma\Ind\Rep(Q)$. Then $\Hom(S_i, \Sigma\tau
M_x)$ vanishes. 

If
$x\neq \Sigma S_i$, we get $\omega^1=0$ by the universal propery of
Auslander-Reiten triangles and consequently
\begin{align*}
(C_qv^{f_i})_{\sigma x}= \dim\Hom(S_i,\tau M_x)-\dim\Hom(S_i,E)+\dim\Hom(S_i,M)=0.
\end{align*}

% can show $\Hom_{\cD^b(Q)}(S_i,E)=0$
For $x=\Sigma S_i$, we get $\Cok w^1=0$ and
consequently 
\begin{align*}
  (C_qv^{f_i})_{\sigma \Sigma S_i}=\dim\Hom(S_i,\tau\Sigma S_i)=1.
\end{align*}
\end{proof}

\subsection{Proof of the one-half quantum group}

We prove Proposition \ref{prop:local_change} in this subsection, which
tells us that the study of cyclic quiver varieties $\cycAffQuot(w)$,
$w\in W^+$ can be reduced to the study of the graded quiver varieties
$\grAffQuot(w)$ for a generic choice of $q$. More precisely, we show
that these cyclic quiver varieties are free of ``wrapping paths'', \cf Example \ref{eg:wrapping_path_free}.

Proposition \ref{prop:local_change} allows us to
translate the results obtained in \cite{HernandezLeclerc11}
\cite{LP_rep_alg} for the latter varieties into Proposition
\ref{prop:generator} Theorem \ref{thm:isom_coordinate_ring}. For
completeness, we give a sketch of the proofs.

\begin{Prop}\label{prop:local_change}
  For any $w\in W^+$, the cyclic quiver variety $\cycAffQuot(w)$ is
  isomorphic to the graded quiver variety $\grAffQuot(w)$.
\end{Prop}
 have studied the
graded quiver varieties $\grAffQuot(w)$ for $w\in w^+$. By Proposition
\ref{prop:local_change}, their results can be used for the the cyclic quiver variety
$\cycAffQuot(w)$, $w\in W^+$.

We give an example to show how a ``wrapping path'' vanishes.
\begin{Eg}\label{eg:wrapping_path_free}
Let us look at Figure \ref{fig:A_2_rep_space}. For any given $w\in
W^+$, take any composition of irreducible morphisms which only passes
through the vertices $x\in\sigma \hat{I}$ or $\sigma S_i$, $i\in I$,
with the ending points of the type $\sigma S_i$.

For
example, we can take a composition $p$, such that the sequence of the vertices it passes through is $(\sigma
S_2, S_1, \sigma S_1, \Sigma P_2, \Sigma S_1, S_2, \sigma S_2)$. Then
$p$ horizontally wraps the figure, in the sense that the heights of
these vertices occupy the whole cyclic group $\langle \epsilon \rangle$. 

Take the factor $p'$ of $p$ corresponding to the subsequence
$(\Sigma P_2, \Sigma S_1, S_2)$. Because of the relations $\mu=0$ and
that $w$ is concentrated on $\sigma S_i$, $i\in I$, $p'$
corresponds to a morphism from $\Sigma P_2$ to $S_2$ in
$(\cD^b(Q))\op$. But such a morphism must vanish. Therefore, $p'$ and
$p$ vanish.
\end{Eg}

\begin{proof}[Proof of Proposition \ref{prop:local_change}]
We shall use the notions of Nakajima categories in the sense of
\cite{KellerScherotzke2013}. Let $\cR$ the mesh category associated with
a generic $q\in \C^*$ and $\cR^\epsilon$ the mesh category
associated with $\epsilon$. Let $\cS$ denote the singular Nakajima category which is generated by
the objects $\sigma x$, $x\in \Ind\Rep(Q)$, in the mesh category
$\cR$. Similarly, define the singular Nakajima category $\cS^\epsilon$
as a subcategory of $\cR^\epsilon$.

By our convention, an $\cS$-module is a functor from $\cS$ to the
category of complex vector spaces. Then the variety of $w$-dimensional
$\cS$-module, denoted by $\Rep(\cS,w)$, is isomorphic $\grAffQuot(w)$, \cf \cite{KellerScherotzke2013} \cite{LP_rep_alg}, and $\cycAffQuot(w)$
is a closed subvariety in $\Rep(\cS^\epsilon,w)$. Notice that, we can naturally embed $\grAffQuot(w)$ into $\cycAffQuot(w)$. Therefore, to verify the proposition, it suffices to show that $\cS$
and $\cS^\epsilon$ are equivalent. Its proof consists of the following
two steps.

% \begin{comment}
%   pf: for generic $q\in \C^*$ and any $l$-dominant $(v,w)$, if $w\in
%   W+$, then we have $v\in V^+$. So we can naturally embed
%   $\grProjQuot(v,w)$, $(v,w)$ $l$-dominant, into
%   $\cycProjQuot(v,w)$.
% \end{comment}

(i) For any two modules $ x, y\in \Ind\Rep(Q)$,
let $p$ be any composition of irreducible morphisms in $\cR^\epsilon$, such that $p$ starts
from $\sigma x$, ends at $\sigma y$, and does not pass any object
$\sigma z\in \cR^\epsilon$ with $z\in \Sigma \Ind\Rep(Q)$ in its
definition. Let $(\ui,\ua)=((i_0,a_0),\ldots,(i_r,a_r))$ denote the
sequence of the objects that $p$ passes through, where
$(i_0,a_0)=\sigma x$, $(i_r,a_r)=\sigma y$, $r\in \N$. Notice that our
convention of $\cR^\epsilon$ implies $a_{t+1}=a_t*\epsilon^{-1}$, $\forall 0\leq t\leq r-1$. By abuse
of notation, let $\ua$ also denote the set $\{a_t|0\leq t\leq r\}$.

% The group
% $\langle q\rangle$ can be naturally viewed as the universal covering of $\langle \epsilon
% \rangle$. Denote the map sending $\epsilon^d$ to $q^d$, $0\leq d<h$ by
% $\iota$. We can then lift the sequence $p$ into the sequence $\tp$ in
% $I\times \langle q\rangle$, such that $\tp=((i_0,\ta_0),\ldots
% (i_r,\ta_r))$ satisfies $\ta_0=\iota(a_0)$,
% $\ta_{t+1}=\ta_t*q^{-1}$. Then $\tp$ corresponds to a unique
% composition of irreducible morphisms in the mesh category $\cR$.

Assume the sequence $(\ui,\ua)$ contains some object outside $\Ind\Rep(Q)$. We want to
show that the morphism $p$ factors through some object in $\sigma\Sigma\Ind\Rep(Q)$.

First notice that the sequence $(\ui,\ua)$ must
contain a consecutive subsequence $(\ui',\ua')$ from
$\Sigma x'$ to $y'$, where $x',y'$ are some indecomposable injective
$\Rep(Q)$. We can require $(\ui',\ua')$ to be small in the sense
that $\ua'\neq \langle \epsilon \rangle$. The factor of $p$ associated
with the small subsequence $(\ui',\ua')$ is denoted by $p'$.

Define the subcategory $\cX^\epsilon$ of $\cR^\epsilon$ such that its
set of objects is $\{I_i,i\in I \}\sqcup\Sigma\Ind \Rep(Q)\sqcup
\sigma\Sigma\Ind \Rep(Q)$ and its morphisms are generated by the irreducible
morphism among these objects in the mesh category
$\cR^\epsilon$. Define the subcategory $\cX$ of $\cR$ similarly. By
comparing the mesh relations we see the two subcategories are
equivalent. Associate to these categories their quotients
$\underline{\cX}$ and $\underline{\cX^\epsilon}$ by sending all the
morphisms factoring through $\sigma\Sigma\Ind \Rep(Q)$ to $0$. Then
the quotient categories are still equivalent.

Notice that $\underline{\cX}$ is equivalent to a subcategory of $(\cD^b(Q))\op$. Therefore, all
morphisms in $\underline{\cX}$ from $\Sigma x'$ to $y'$, $x',y'\in \{I_i,i\in I \}$
vanish. Because the subsequence $(\ui',\ua')$ is small, the morphism
$p'$ is well defined on $\cX^\epsilon$. It follows that $p'=0$ in
$\underline{\cX^\epsilon}$. Therefore, in the category $\cR^\epsilon$, $p'$ and $p$ factors through the objects
of $\sigma\Sigma\Ind\Rep(Q)$.

(ii) By (i), we deduce that the singular category $\cS^\epsilon$ is the subcategory
of $\cR^\epsilon$ whose set of objects is $\sigma \hat{I}$ and whose
morphisms are linear combinations of compositions of the irreducible maps among the elements
in $\sigma
\hat{I}\cup \ind \C Q$. By comparing the mesh relations, we see
$\cS^\epsilon$ is equivalent to $\cS$.
\end{proof}

% \begin{draft}
%   Compare the w-dimensions below to simplify the proof.
% \end{draft}

\begin{Prop}[\cite{HernandezLeclerc11} \cite{LP_rep_alg}]\label{prop:w_gr_elements}
For any $w\in W^S$, $\grAffQuot(w)$ is isomorphic to $\Rep(Q,\sum_i
w_{\sigma S_i} e_i)$. Moreover, the nonempty regular strata are in
bijection with the orbits of  $\Rep(Q,\sum_i
w_{\sigma S_i} e_i)$, and consequently, in bijection with the dual
canonical basis
elements of $U_t(\mfr{n}^+)$ with the homogeneous degree $\sum_i
w_{\sigma S_i}\alpha_i$. 
\end{Prop}

\begin{proof}[{Proof of Proposition \ref{prop:generator} Theorem \ref{thm:isom_coordinate_ring}}]
(i) We first prove Proposition \ref{prop:generator}(1) and Theorem
\ref{thm:isom_coordinate_ring}(1). 

$R^+$ is generated by
$\can(0,e_{\sigma S_i})$, $i\in I$. We will show that these generators
satisfies the quantum Serre relations in Proposition
\ref{prop:Serre_relation}. Then the identification between the
Chevalley generators induces a surjective map from
$\Ut(\mfr{n}^+)_{\Q(t^\Hf)}$ to $R^+_{\Q(t^\Hf)}$. It remains to check
that the two $W^S$-graded algebra have the same graded dimension, which follows
from Propositions \ref{prop:local_change} \ref{prop:w_gr_elements}.

(ii) The proof of Proposition \ref{prop:generator}(2) is the same as
in (i).

(iii) The claim of Theorem
\ref{thm:isom_coordinate_ring}(2) is a consequence of Theorem
\ref{thm:isom_coordinate_ring}(1). More details can be found in
the proof of \cite[Theorem
6.1(2)]{HernandezLeclerc11}).
\end{proof}
\begin{Rem}
  % In this paper, the proof of \ref{thm:isom_coordinate_ring}
  % consists of Proposition \ref{prop:isom_Grothendieck_ring} and the
  % verification of the quantum Serre relations. Therefore, it can be
  % viewed as a proof different to that of \cite[Theorem
  % 6.1]{HernandezLeclerc11} which depends on the comparison of the
  % quantum $T$-systems involved.
One can also obtain Theorem \ref{thm:isom_coordinate_ring} by
identifying $\otimes$ on the subalgebra $R^+$
with the multiplication of the deformed Grothendieck ring in \cite{Hernandez02} \cite{HernandezLeclerc11} and applying the
result of \cite{HernandezLeclerc11}. The identification will be
discussed in the last section.
\end{Rem}

% \begin{proof}
% To identify the $t$-deformations, it suffices
% to verify the following equality
% for any $l$-dominant pairs $(v^1,w^1)$, $(v^2,w^2)$, $v^1,v^2\in V^+$,
% $w^1,w^2\in W^S$:
% \begin{align}
%   \begin{split}
%   -d((v^1,w^1),(v^2,w^2))+d(v^2,w^2),(v^1,w^1))-\langle
%   w^1,w^2\rangle_a\\=\cN(w^1-C_qv^1,w^2-C_q v^2),
% \end{split}
% \end{align}
% where the bilinear form $\cN$ is used to construct the deformed
% Grothendieck ring $\cK^Q_t$.

% Let us denote $m^1=(v^1,w^1), m^2=(v^2,w^2), u^1=w^1-C_q v^1,
% u^2=w^2-C_q v^2$. Then the bilinear form $\cN$ in \cite[(6)]{HernandezLeclerc11} can be written as
% \begin{align}
%   \begin{split}
%   \cN(u^1,u^2)=& C_q^{-1}u^1\cdot \sigma^* u^2-C_q^{-1}u^1\cdot
%   (\sigma^{-1})^* u^2\\& +C_q^{-1}u^2\cdot (\sigma^{-1})^* u^1-
%   C_q^{-1}u^2\cdot \sigma^* u^1.
% \end{split}
% \end{align}

% The computation in \cite[Section 4.2.1]{QinThesis} (\cf also
% \cite{Qin12}) implies
% \begin{align}
%   \LHS=(\sigma^{-1})^*u^1\cdot C_q^{-1}u^2-(\sigma^{-1}) ^*u^2\cdot
%   C_q^{-1} u^1-\langle w^1,w^2\rangle_a.
% \end{align}
% So it remains to verify
% \begin{align*}
%   \langle w^1,w^2\rangle_a=-C_q^{-1}u^1\cdot \sigma^* u^2+C_q^{-1}u^1\cdot
%   (\sigma^{-1})^* u^2,
% \end{align*}
% for any $w^1,w^2\in W^+$.

% \end{proof}
% \begin{draft}
%   Specializing $t=1$ we obtain both algebra have the same rank. To
%   obtain the statement, we can calculate the q-Serre relations. But to
%   draw relations with quantum coordinate ring, we use the next Proposition?
% \end{draft}

\subsection{Proof of the triangular decomposition}

Denote the extension $(\ )\otimes_\N \Z$ by $(\ )_\Z$.

For any $i,j\in I$, we have
\begin{align}\label{eq:injective_coeff}
  \begin{split}
(v^{f_i})_{I_j}=(v^{\Sigma f_i})_{\Sigma I_j}=\delta_{ij},\\
(v^{f_i})_{\Sigma I_j}=(v^{\Sigma f_i})_{I_j}=0.
\end{split}
\end{align}
The following lemma follows as a consequence.
\begin{Lem}\label{lem:decomposition_V}
$\N^{\sigma\hat{I}}$ is a subset of $V_\Z^+\oplus V^0\oplus V_\Z^-$.
\end{Lem}
\begin{proof}
  For any $v\in \N^{\sigma\hat{I}}$, we define $v^0=\sum b_i
  v^{f_i}+\sum b_i' v^{\Sigma f_i}$ such that $b_i=v(I_i)$ and
  $b_i'=v(\Sigma I_i)$. Define $v^+$ to be the restriction of $v-v^0$
  on $V^+_\Z$ and $v^-$ the restriction of $v-v^0$ on
  $V^-_\Z$. \eqref{eq:injective_coeff} guarantees that
  $v=v^++v^0+v^-$ is our desired decomposition.
\end{proof}

Denote the projections of $\N^{\sigma \hat{I}}$ to the three
summands in Lemma \ref{lem:decomposition_V} by $\pr^+$, $\pr^0$,
$\pr^-$ respectively. The following result is essentially known by \cite{LP_rep_alg}.
\begin{Prop}\label{prop:dominant_combinatorics}
  For any $w\in W^S_ \Z$, $v\in V^+_\Z$, if $w-C_q v\geq 0$,
  then $v\in V^+$, $w\in W^S$.
\end{Prop}
\begin{proof}

 To verify the statement, by Proposition \ref{prop:local_change}, we can work in the case
where $q$ is not a root of unity instead. We then prove
it by using Theorem \ref{thm:abelian_representative} and \cite[Theorem
3.14]{LP_rep_alg}.

Let $\hat{A}$ denote the repetitive algebra of $A=\C Q$. By using
Syzygy functors in $\mod\hat{A}$ (the category of left $\hat{A}$-modules), we can
identify the sets $W^S$, $V^+$ with subsets of $\N^{\psi(\proj
  \hat{A})}$, $\N^{\psi(\Ind \mod\hat{A}-\proj\hat{A})}$ studied in \cite[Section 3.1]{LP_rep_alg}, \cf \cite[Rem
3.17]{LP_rep_alg}. From now on, we work in the context of
\cite{LP_rep_alg}. 

Denote $\tw=w-C_q v$. By Theorem
\ref{thm:abelian_representative}, there exists a unique $l$-dominant pair
$(v',w')$, $v'\in V^+$, $w'\in W^S$, such that $w'-C_q v'=\tw$.

To any
$\hat{A}$-module $N$ of dimension $w^N$, we associate the module
$\underline{N}$ defined in \cite[Lemma 3.12]{LP_rep_alg}. Its
dimension will be denoted by $(v^N,w^N)\in \N^{\psi(\Ind \mod\hat{A}-\proj\hat{A})}\times\N^{\psi(\proj
  \hat{A})}$. Moreover, the pair $(v^N,w^N)$ is $l$-dominant. Notice that we always
have $w^P-C_q v^P=0$ for any projective $P$ in $\proj\hat{A}$. Let us take
some projective $P$ with its dimension big enough, such that $v+v^P\geq 0$,
$w+w^P\geq 0$. Then $(v+v^P,w+w^P)$ is an $l$-dominant pair.  By
\cite[Theorem 3.14]{LP_rep_alg}, it determines the isoclass of an $\hat{A}$ module $N$ of dimension $w^N$ such that
$(v^N,w^N)=(v+v^P,w+w^P)$.

% Because $w-C_q v\geq 0$, by the proof of
% \cite[Theorem 3.14(ii)]{LP_rep_alg}, there exists some projective
% module $P$ such that $(v+v^P,w+w^P)$ equals $(v^N,w^N)$ for some
% $\hat{A}$-module $N$. 

Since both $w'$ and $w^N$ are contained in $\N^{\psi(\proj \hat{A})}$ and
$w'-C_q v'=w^N-C_q v^N$, by
\cite[Section 4.3]{LP_rep_alg}, there exists some projective $\hat{A}$-modules $P^1$, $P^2$ such that
$(v'+v^{P^1},w'+w^{P^1})=(v^N+v^{P^2},w^N+w^{P^2})$. Then we have
$(v',w')=(v+v^P+v^{P^2}-v^{P^1},w+w^P+w^{P^2}-w^{P^1})$. Notice that $w^P+w^{P^2}-w^{P^1}$ is
not contained $W^S\otimes \Z$ unless it vanishes (in other words, $P\oplus P^2=P^1$). Because
$w'$ and $w$ are contained in $W^S\otimes\Z$. It follows that
$w^P+w^{P^2}-w^{P^1}$ vanishes. Therefore, we
obtain $(v',w')=(v,w)$.
\end{proof}

% \begin{draft}
%   Maybe a combinatorial proof exists. See examples for linear $A$ type: $(C_q v)_{x\neq S}\geq 0$, $v\in
%   V^+$, means $v\geq 0$. 

% Maybe the most natural proof is via an adaption of the representation
% theory to the root category.
% \end{draft}

\begin{Prop}\label{prop:dominant_decomposition}
 For any $w\in W^S\oplus W^{\Sigma S}$, $v\in \N^{\sigma\hat{I}}$,
 assume that the pair $(v,w)$ is $l$-dominant, then we have a unique decomposition of $(v,w)$ into $l$-dominant pairs
$(v^+,w^+)$, $(v^0,w^0)$, $(v^-,w^-)$, such that $v^+\in V^+$, $v^0\in
V^0$, $v^-\in V^-$, $w^+\in W^S$,
$w^0\in W^0$, $w^-\in W^{\Sigma S}$, $v^++v^0+v^-=v$, $w^++w^0+w^-=w$,
and $w^0-C_q v^0=0$,.
\end{Prop}
\begin{proof}
Denote the projections of $\Z^{\hat{I}}$ onto $W^+_\Z$ and $W^-_\Z$ by $\pi^+$ and $\pi^-$ respectively.

We first take $v^0=\pr^0 V^0$ and denote the natural decomposition of $v^0$ in $V^0$ by $v^0=\sum_i b_i
v^{f_i}+\sum_{i}b_i' v^{\Sigma f_i}$. Further define
$w^0=\sum_i(b_i+b_i')w^{f_i}$. Define $v^+=\pr^+ v$, $v^-=\pr^- v$.

We have $\pi^+C_q v^-=0$. Therefore $C_q v^+$ equals $\pi^+C_q (v-v^0)$. Because $w^0-C_q v^0=0$, we have
$\pi^+(w-w^0)-C_q v^+=\pi^+(w-w^0)-\pi^+C_q (v-v^0)=\pi^+(w-C_qv)\geq 0$. By Proposition
\ref{prop:dominant_combinatorics}, $(v^+,\pi^+(w-w^0))$ is an $l$-dominant
pair with $v^+\in V^+$, $\pi^+(w-w^0)\in W^S$. Similarly, we obtain the
$l$-dominant pair $(v^-,\pi^-(w-w^0))$ with $v^-\in V^-$, $\pi^-(w-w^0)\in
W^{\Sigma S}$. 

Define $w^+=\pi^+(w-w^0)$, and $w^-=\pi^-(w-w^0)$. Then the
decomposition $(v,w)=(v^+,w^+)+(v^0,w^0)+(v^- ,w^-)$ satisfies the
conditions we impose.

Finally, let us prove the uniqueness. Lemma \ref{lem:decomposition_V} implies the
decomposition $v=v^++v^0+v^-$ is unique. Then $w^0$ is determined by
$v^0$. It follows that the decomposition $w=w^++w^0+w^-$ is unique.
\end{proof}

\begin{proof}[{proof of Theorem \ref{thm:triangular_decomposition}}]
  As a consequence of Proposition
  \ref{prop:dominant_decomposition}, \cf also Proposition \ref{prop:finite_dominant}, for any $w\in W^S\oplus W^{\Sigma S}$, there exists finitely many $v$ such that $(v,w)$ is
$l$-dominant. Combining Proposition \ref{prop:dominant_decomposition} and
  \eqref{eq:multiplication_can}, we get Theorem
  \ref{thm:triangular_decomposition} by induction on $v$.
\end{proof}

\subsection{Proof of the main results}

We explicitly calculate some relations of the generators of $(R,\tOtimes)$.
\begin{Prop}\label{prop:EK}
For any $i,j\in I$, we have
\begin{align}
  \can(0,e_{\sigma S_i})\tOtimes
  \can(v^{f_j},w^{f_j})&=t^{2\langle S_i,S_j\rangle}\can(v^{f_j},w^{f_j})\tOtimes
  \can(0,e_{\sigma S_i})\\
\can(0,e_{\sigma S_i})\tOtimes
    \can(v^{\Sigma f_j},w^{f_j})&=t^{-2\langle S_j,S_i\rangle}\can(v^{\Sigma f_j},w^{ f_j})\tOtimes
    \can(0,e_{\sigma  S_i})\\ 
  \can(0,e_{\sigma\Sigma S_i})\tOtimes
  \can(v^{f_j},w^{f_j})&=t^{-2\langle S_j, S_i\rangle}\can(v^{f_j},w^{f_j})\tOtimes
  \can(0,e_{\sigma \Sigma S_i})\\
  \can(0,e_{\sigma \Sigma S_i})\tOtimes
  \can(v^{\Sigma f_j},w^{ f_j})&=t^{2\langle S_i,S_j\rangle }\can(v^{\Sigma
    f_j},w^{f_j})\tOtimes
  \can(0,e_{\sigma \Sigma S_i}).
\end{align}
 \end{Prop}
 \begin{proof}
   We shall further prove that for each relation, either side consists of only the leading
   term\footnote{This situation is usually called \emph{special} or
     \emph{affine-minuscule}.} when it decomposes via
   \eqref{eq:multiplication_can}.

(i) We start by verifying the first relation. 

By Proposition \ref{prop:dominant_decomposition}, any $l$-dominant pair
$(v,e_{\sigma S_i}+w^{f_j})$ decompose into the sum of three
$l$-dominant pairs $(v^+,w^+)$, $(v^0,w^0)$, $(v^-,w^-)$. Applying
\eqref{eq:multiplication_can} to LHS of the first relation, we see the
term $\can(v,e_{\sigma S_i}+w^{f_j})$ has nonzero coefficients only if $v\geq
v^{f_j}$. So we obtain $v^0\geq v^{f_j}$ and, consequently, $w^0\geq
w^{f_j}$. It follows that the only possible decomposition of $w$ is
$w^+=e_{\sigma S_i}$, $w^0=w^{f_j}$, $w^-=0$. Consequently, $v$ has
only one possible decomposition: $v^+=v^-=0$,
$v^0=v^{f_j}$. Therefore, both sides are just multiples of the leading
term $\can(v^{f_j},e_{\sigma S_i}+w^{f_j})$.

The claim follows from the calculation of the $t$-power for the coefficients of the leading terms.
\begin{align*}
  d((0,e_{\sigma S_i}),(v^{f_j},w^{f_j}))&=e_{\sigma S_i}*\sigma^*
  v^{f_j}=v^{f_j}(\sigma^2 S_i)=\dim\Hom_{\cD^b(Q)}(S_j,\tau S_i)\\
&=\dim\Hom_{\cD^b(Q)}(S_i,\Sigma S_j),\\
d((v^{f_j},w^{f_j}),(0,e_{\sigma S_i}))&=v^{f_j}*\sigma^* e_{\sigma
  S_i}=v^{f_j}*e_{S_i}\\
&=\dim\Hom_{\cD^b(Q)}(S_j,S_i)=\dim\Hom_{\cD^b(Q)}(S_i,S_j).% \\
% d((0,e_{\sigma S_i}),( v^{\Sigma f_j},w^{f_j}))&=v^{f_j}(\Sigma \tau
% S_i). \\
% &=\dim\Hom _{\cD^b(Q)}(S_i,S_j)=\dim\Hom_{\cD^b(Q)}(S_i,S_j),\\
% d((v^{\Sigma f_j},w^{f_j}),(0,e_{\sigma S_i}))&=\Sigma^* v^{f_j}*
% \sigma^* e_{\sigma S_i}\\
% &=\dim\Hom_{\cD^b(Q)}(S_j,\Sigma S_i).
\end{align*}
% By applying $\Sigma^*$, we obtain
% \begin{align*}
%   d((0,e_{\sigma\Sigma S_i}),(v^{\Sigma f_j},w^{f_j}))=d((0,e_{\sigma
%     S_i}),(v^{f_j},w^{f_j})),\\
% d((v^{\Sigma f_j},w^{f_j}),(0,e_{\sigma \Sigma S_i}))=d((v^{f_j},w^{f_j}),(0,e_{\sigma S_i})).
% \\
% d((0,e_{\sigma\Sigma  S_i}),(v^{\Sigma f_j},w^{f_j}))=d((0,e_{\sigma  S_i}),( v^{f_j},w^{f_j})),\\
% d((v^{\Sigma f_j},w^{f_j}),(0,e_{\sigma \Sigma
%   S_i}))=d((v^{f_j},w^{f_j}),(0,e_{\sigma  S_i})).
%\end{align*}
% The claim follows.% The $t$-powers in the proposition follows from direct computation.

(ii) The verification of the second relation is similar.

By Proposition \ref{prop:dominant_decomposition}, any $l$-dominant pair
$(v,e_{\sigma S_i}+w^{f_j})$ decompose into the sum of three
$l$-dominant pairs $(v^+,w^+)$, $(v^0,w^0)$, $(v^-,w^-)$. Applying
\eqref{eq:multiplication_can} to LHS of the second relation, we see the
term $\can(v,e_{\sigma S_i}+w^{f_j})$ has nonzero coefficients only if $v\geq
v^{\Sigma f_j}$. So we obtain $v^0\geq v^{\Sigma f_j}$ and, consequently, $w^0\geq
w^{f_j}$. It follows that the only possible decomposition of $w$ is
$w^+=e_{\sigma S_i}$, $w^0=w^{f_j}$, $w^-=0$. Consequently, $v$ has
only one possible decomposition: $v^+=v^-=0$,
$v^0=v^{\Sigma f_j}$. Therefore, both sides are just multiples of the leading
term $\can(v^{\Sigma f_j},e_{\sigma S_i}+w^{f_j})$.

The claim follows from the calculation of the $t$-power for the coefficients of the leading terms.% It remains to calculate the coefficients of the leading terms. For the
% first two relations, we compute the following bilinear forms.
\begin{align*}
%   d((0,e_{\sigma S_i}),(v^{f_j},w^{f_j}))&=e_{\sigma S_i}*\sigma^*
%   v^{f_j}=v^{f_j}(\sigma^2 S_i)=\dim\Hom_{\cD^b(Q)}(S_j,\tau S_i)\\
% &=\dim\Hom_{\cD^b(Q)}(S_i,\Sigma S_j),\\
% d((v^{f_j},w^{f_j}),(0,e_{\sigma S_i}))&=v^{f_j}*\sigma^* e_{\sigma
%   S_i}=v^{f_j}*e_{S_i}\\
% &=\dim\Hom_{\cD^b(Q)}(S_j,S_i)=\dim\Hom_{\cD^b(Q)}(S_i,S_j),\\
d((0,e_{\sigma S_i}),( v^{\Sigma f_j},w^{f_j}))&=v^{f_j}(\Sigma \tau
S_i)\\
&=\dim\Hom _{\cD^b(Q)}(S_i,S_j)=\dim\Hom_{\cD^b(Q)}(S_i,S_j),\\
d((v^{\Sigma f_j},w^{f_j}),(0,e_{\sigma S_i}))&=\Sigma^* v^{f_j}*
\sigma^* e_{\sigma S_i}\\
&=\dim\Hom_{\cD^b(Q)}(S_j,\Sigma S_i).
\end{align*}
% By applying $\Sigma^*$, we obtain
% \begin{align*}
%   d((0,e_{\sigma\Sigma S_i}),(v^{\Sigma f_j},w^{f_j}))=d((0,e_{\sigma
%     S_i}),(v^{f_j},w^{f_j})),\\
% d((v^{\Sigma f_j},w^{f_j}),(0,e_{\sigma \Sigma S_i}))=d((v^{f_j},w^{f_j}),(0,e_{\sigma S_i})),\\
% d((0,e_{\sigma\Sigma  S_i}),(v^{\Sigma f_j},w^{f_j}))=d((0,e_{\sigma  S_i}),( v^{f_j},w^{f_j})),\\
% d((v^{\Sigma f_j},w^{f_j}),(0,e_{\sigma \Sigma
%   S_i}))=d((v^{f_j},w^{f_j}),(0,e_{\sigma  S_i})).
% \end{align*}
% The $t$-powers in the proposition follows from direct computation.
% The claim follows.

(iii)(iv) The automorphism $\Sigma^*$ on the dimension vectors $v,w$
induces isomorphisms of cyclic quiver varieties, which are compatible
with the (twisted) resitriction functors $\tRes^{w}_{w^1,w^2}$, $\res^w_{w^1,w^2}$, as well as the bilinear form $d(\ ,\ )$. Therefore,
the first two relations imply
\begin{align*}
  \can(0,\Sigma^*e_{\sigma S_i})\tOtimes
  \can(\Sigma^* v^{f_j},\Sigma^*w^{f_j})&=t^{2\langle S_i,S_j\rangle}\can(\Sigma^*v^{f_j},\Sigma^*w^{f_j})\tOtimes
  \can(0,\Sigma^*e_{\sigma S_i}),\\
\can(0,\Sigma^*e_{\sigma S_i})\tOtimes
    \can(\Sigma^*v^{\Sigma f_j},\Sigma^*w^{f_j})&=t^{-2\langle S_j,S_i\rangle}\can(\Sigma^*v^{\Sigma f_j},\Sigma^*w^{ f_j})\tOtimes
    \can(0,\Sigma^*e_{\sigma  S_i}).
\end{align*}
These are just the fourth and third relations respectively.
\end{proof}

\begin{Prop}\label{prop:EF}
In $(R,\tOtimes)$, for any $i,j\in I$, we have
\begin{align}\label{eq:EF}
  [\can(0,e_{\sigma S_i}),\can(0,e_{\sigma\Sigma
    S_j})]=\delta_{ij}(t-t^{-1})(\can(v^{f_i},w^{f_i})-\can(v^{\Sigma f_i},w^{\Sigma
    f_i})).
\end{align}
\end{Prop}
\begin{proof}
  (i) Assume $i\neq j$. We deduce from Proposition
  \ref{prop:dominant_decomposition} that the only $l$-dominant pair
  $(v,e_{\sigma S_i}+e_{\sigma \Sigma S_j})$ is given by $v=0$. Let us
  calculate the bilinear forms:
  \begin{align*}
    d((0,e_{\sigma S_i}),(0,e_{\sigma \Sigma S_j}))=0,\\
    d((0,e_{\sigma \Sigma S_j}),(0,e_{\sigma S_i}))=0,\\
  \end{align*}
The statement follows.

(ii) Assume $i=j$. By Proposition \ref{prop:dominant_decomposition},
the only $l$-dominant pairs $(v,w^{f_i})$ are given by $v=0,\ v^{f_i},\
v^{\Sigma f_i}$. Then we have
\begin{align*}
  \pi(0,e_{\sigma S_i})&=\cL(0,e_{\sigma S_i}),\\
\pi(0,e_{\sigma \Sigma S_i})&=\cL(0,e_{\sigma \Sigma S_i}),\\
\pi(0,w^{f_i})&=\cL(0,w^{f_i}),
\end{align*}
\begin{align}\label{eq:compute_decompose}
\pi(v^{f_i},w^{f_i})&=\cL(v^{f_i},w^{f_i})+a_{v^{f_i},0;w^{f_i}}\cL(0,w^{f_i}),
\end{align}
\begin{align}
\pi(v^{\Sigma f_i},w^{f_i})&=\cL(v^{\Sigma f_i},w^{f_i})+a_{v^{\Sigma f_i},0;w^{f_i}}\cL(0,w^{f_i})
\end{align}

Notice that, by the definition of the GIT quotient and
the dimension vector $v^{f_i}$, the GIT quotient $\cycProjQuot(v^{f_i},e_{\sigma S_i})$ is a
point, \cf Example \ref{eg:explicit_EF}. Therefore, it is isomorphic
to the variety $\cycAffQuot(e_{\sigma S_i})=\cycAffQuot(0,e_{\sigma
    S_i})$. So we get $\pi(v^{f_i},e_{\sigma S_i})=1_{\set{0}}=\cL(0,e_{\sigma
      S_i})$. Similarly, we obtain
    \begin{align*}
      \pi(v^{\Sigma f_i},e_{\sigma S_i})&=\pi(v^{f_i},e_{\sigma
        S_i})=\cL(0,e_{\sigma S_i}),\\
      \pi(v^{\Sigma f_i},e_{\sigma \Sigma S_i})&=\pi(v^{f_i},e_{\sigma
        \Sigma S_i})=\cL(0,e_{\sigma \Sigma S_i}).
    \end{align*}

Notice that for any $v\in V^+$, $\cycProjQuot(v,e_{\sigma\Sigma S_i})$
is empty unless $v=0$. Therefore, by applying the restriction functor to LHS of
\eqref{eq:compute_decompose}, we obtain% \footnote{Notice that the degree shifts
  % are anti-symmetric. A simple way to calculate them is to replace $e_{\sigma\Sigma S_i}$ by $w^{f_i}$ and
  % use the calculation in Proposition \ref{prop:EK}.}
\begin{align*}
  \tRes^{w^{f_i}}_{e_{\sigma S_i}, e_{\sigma \Sigma
    S_i}}&\pi(v^{f_i},w^{f_i})\\
&=\oplus_{v^1,v^2}\pi(v^1,e_{\sigma S_i})\boxtimes
  \pi(v^2,e_{\sigma \Sigma S_i})[d((v^2,e_{\sigma \Sigma S_i}),(v^1,e_{\sigma S_i}))-d((v^1,e_{\sigma S_i}),(v^2,e_{\sigma \Sigma S_i}))]\\
&=\pi(v^{f_i},e_{\sigma S_i})\boxtimes
  \pi(0,e_{\sigma \Sigma S_i})[1]\\
&=\cL(0,e_{\sigma S_i})\boxtimes
  \cL(0,e_{\sigma \Sigma S_i})[1].
\end{align*}
In other word, the external tensor $\cL(0,e_{\sigma S_i})\boxtimes
  \cL(0,e_{\sigma \Sigma S_i})$ will have coefficient $t$ when we
  apply the restriction functor to $\LHS$ of \eqref{eq:compute_decompose}.

The following relation is obvious by definition
\begin{align*}
  \tRes^{w^{f_i}}_{e_{\sigma S_i},e_{\sigma \Sigma
    S_i}}\pi(0,w^{f_i})&=\pi(0,e_{\sigma S_i})\boxtimes
  \pi(0,e_{\sigma \Sigma S_i})\\
&=\cL(0,e_{\sigma S_i})\boxtimes
  \cL(0,e_{\sigma \Sigma S_i}).
\end{align*}
Therefore, by applying the restriction functor to the $\RHS$ of
\eqref{eq:compute_decompose}, the second term will contribute an
external tensor $\cL(0,e_{\sigma S_i})\boxtimes
  \cL(0,e_{\sigma \Sigma S_i})$ with the bar-invariant coefficient
  $a_{v^{f_i},0;w^{f_i}}$. Because the coefficients appearing
  under the restriction functor are nonnegative, in order for this
  external product to have coefficient $t$ as in $\LHS$ of \eqref{eq:compute_decompose}, we
must have \footnote{It might be possible to verify this statement by studying the fiber of
  $\cycProjQuot(v^{f_i},w^{f_i})$ over the origin of
  $\cycAffQuot(w^{f_i})$, \cf Example \ref{eg:explicit_EF}.} $a_{v^{f_i},0;w^{f_i}}=0$.

Similarly, we have $a_{v^{\Sigma f_i},0;w^{f_i}}=0$ and 
\begin{align*}
  \tRes^{w^{f_i}}_{e_{\sigma S_i},e_{\sigma \Sigma
    S_i}}\pi(v^{\Sigma f_i},w^{f_i})=\cL(0,e_{\sigma S_i})\boxtimes
  \cL(0,e_{\sigma \Sigma S_i})[-1].
\end{align*}

By taking the dual of the restriction functor, we obtain the following
equation
\begin{align}\label{eq:EF_1}
  \can(0,e_{\sigma S_i})\tOtimes \can(0,e_{\sigma \Sigma
    S_i})=\can(0,w^{f_i})+t\can(v^{f_i},w^{f_i})+t^{-1}\can(v^{\Sigma f_i},w^{f_i}).
\end{align}
Similarly, by using the isomorphisms of cyclic quiver varieties
induced by the automorphism $\Sigma^*$ on the dimension vectors
$v,w$, we obtain
\begin{align}\label{eq:EF_2}
  \can(0,e_{\sigma \Sigma S_i})\tOtimes \can(0,e_{\sigma  S_i})=\can(0,w^{f_i})+t\can(v^{\Sigma f_i},w^{f_i})+t^{-1}\can(v^{f_i},w^{f_i}).
\end{align}
The Propositions follows.
\end{proof}

\begin{Eg}\label{eg:explicit_EF}
  Let us continue Example \ref{eg:SL_2} and verify Proposition
  \ref{prop:EF} for the case $\msf{sl}_2$.

In this case, since $I=\set{1}$, we drop the subscript $i$ for
simplicity.

We first consider the decompositions of perverse sheaves (and their shifts)
\begin{align}\label{eq:decompose_S_f}
  \pi(v^f,w^f)=\cL(v^f,w^f)+a_{v^f,0;w^f}\cL(0,w^f),
\end{align}
\begin{align}
  \pi(v^{\Sigma f},w^f)=\cL(v^{\Sigma f},w^f)+a_{v^{\Sigma f},0;w^f}\cL(0,w^f).
\end{align}
In fact, we can compute the coefficients directly as
follows. The smooth quiver variety $\cycProjQuot(v^f,w^f)$ is the $\C^*$-quotient of variety
$\set{(\beta,\alpha)|\C\xla{\beta}\C \xla{\alpha} \C,\ker\beta=0}$
where the torus $\C^*$ naturally acts on $\beta$. Therefore,
it is simply the vector space $\C$. The quiver variety
$\cycAffQuot(v^f,w^f)$ is simply the vector space $\C$. The
projection map from $\cycProjQuot(v^f,w^f)$ to $\cycAffQuot(v^f,w^f)$
sending $(\beta,\alpha)$ to $\beta\alpha$ is an
isomorphism. Therefore, the coefficient $a_{v^f,0;w^f}$ vanishes and
we have $$\pi(v^f,w^f)=\cL(v^f,w^f).$$ The automorphism $\Sigma^*$ on
the dimension vectors $v,w$ induces isomorphisms of quiver
varieties. So we similarly have $a_{v^{\Sigma f},0;w^f}=0$ and
$$\pi(v^{\Sigma f},w^f)=\cL(v^{\Sigma f},w^f).$$ 

% Let us forget this quick calculation of the coefficients and return to the part (ii) of the proof for Proposition
% \ref{prop:EF}. The cyclic quiver varieties $\cycProjQuot(v^f,e_{\sigma
% S})$ is the $\C^*$-quotient of the variety
% $\set{\beta|\C\xla{\beta}\C,\ker \beta=0}$. Therefore, it is a point. We get
% $\pi(v^f,e_{\sigma S})=1_{\set{0}}=\cL(0,e_{\sigma S})$. Similarly,
% $\pi(v^f,e_{\sigma S})=1_{\set {0}}=\cL(0,e_{\sigma \Sigma S})$. Now,
% by applying the restriction functor $\tRes^{w^f}_{e_{\sigma S}
%   e_{\sigma \Sigma S}}$ to both sides of
% \eqref{eq:decompose_S_f}, we obtain that
% \begin{align*}
%   \pi(v^f,e_{\Sigma S})\boxtimes \pi(0,e_{\sigma \Sigma S})[1]=\tRes^{w^f}_{e_{\sigma S}
%   e_{\sigma \Sigma S}}\cL(v^f,w^f)+a_{v^f,0;w^f}\tRes^{w^f}_{e_{\sigma S}
%   e_{\sigma \Sigma S}}\cL(0,w^f).
% \end{align*}
% We have seen that the left hand side is just $\cL(0,e_{\Sigma
%   S})\boxtimes \cL(0,e_{\sigma \Sigma S})[1]$. Because the coefficient of the
% restriction functor should be non-negative, we must have
% $a_{v^f,0;w^f}=0$, otherwise $\RHS$ will be a sum of more than one . Similarly, the coefficient $a_{v^{\Sigma f},0;w^f}$ vanishes.

Also notice that, because the GIT quotient $\cycProjQuot(v^{f
},e_{\sigma S })$ is simply a point, we have $\pi(v^{f },e_{\sigma S
})=1_{\set{0}}=\can(0,e_{\sigma S})$. Similarly, $\pi(v^{\Sigma f },e_{\sigma S
})=1_{\set{0}}=\can(0,e_{\sigma S})$.

Therefore, we obtain 
\begin{align*}
  \tRes^{w^{f}}_{e_{\sigma S },e_{\sigma \Sigma
    S }}&\pi(v^{f },w^{f })\\
&=\oplus_{v^1,v^2}\pi(v^1,e_{\sigma S})\boxtimes
  \pi(v^2,e_{\sigma \Sigma S})[d((v^2,e_{\sigma \Sigma
    S}),(v^1,e_{\sigma S}))-d((v^1,e_{\sigma S}),(v^2,e_{\sigma \Sigma S}))]\\
&=\pi(v^{f },e_{\sigma S })\boxtimes
  \pi(0,e_{\sigma \Sigma S })[1]\\
&=\cL(0,e_{\sigma S })\boxtimes
  \cL(0,e_{\sigma \Sigma S })[1],
\end{align*}
and similarly,
\begin{align*}
  \tRes^{w^{f}}_{e_{\sigma S },e_{\sigma \Sigma S }}\pi(v^{\Sigma f
  },w^{f })&=\cL(0,e_{\sigma S})\boxtimes \cL(0,e_{\sigma \Sigma S
  })[-1].
\end{align*}

By isomorphisms of quiver varieties induced by the automorphism
$\Sigma^*$, the above relations imply
\begin{align*}
  \tRes^{w^{f}}_{e_{\sigma \Sigma S },e_{\sigma  S }}\pi(v^{\Sigma f
  },w^{f })&=\cL(0,e_{\sigma \Sigma S })\boxtimes
  \cL(0,e_{\sigma S })[1],\\
  \tRes^{w^{f}}_{e_{\sigma\Sigma  S },e_{\sigma 
    S }}\pi(v^{ f },w^{f })&=\cL(0,e_{\sigma\Sigma S })\boxtimes
  \cL(0,e_{\sigma S })[-1].
\end{align*}

The following equations are obvious by definition.
\begin{align*}
  \tRes^{w^{f }}_{e_{\sigma S },e_{\sigma \Sigma
    S }}\pi(0,w^{f })=&\cL(0,e_{\sigma S })\boxtimes
  \cL(0,e_{\sigma \Sigma S }),\\
  \tRes^{w^{f }}_{e_{\sigma \Sigma S },e_{\sigma S }}\pi(0,w^{f })=&\cL(0,e_{\sigma\Sigma S })\boxtimes
  \cL(0,e_{\sigma S }).
\end{align*}

Equations \eqref{eq:EF_1} and \eqref{eq:EF_2} are obtained by taking
the dual of the restriction functors $\tRes^{w^{f }}_{e_{\sigma S },e_{\sigma \Sigma
    S }}$ and $\tRes^{w^{f }}_{e_{\sigma \Sigma S },e_{\sigma S }}$.
\end{Eg}

\begin{Prop}\label{prop:KK}
For any $i,j\in I$, we have
\begin{align}
  \can(v^{f_i},w^{f_i})\tOtimes
  \can(v^{f_j},w^{f_j})&=t^{\langle S_i,S_j\rangle -\langle
    S_j,S_i\rangle }\can(v^{f_i}+v^{f_j},w^{f_i}+w^{f_j})\\
  \can(v^{f_i},w^{f_i})\tOtimes
  \can(v^{\Sigma f_j},w^{f_j})&=t^{\langle S_i,S_j\rangle -\langle
    S_j,S_i\rangle }\can(v^{f_i}+v^{\Sigma f_j},w^{f_i}+w^{f_j})\\
  \can(v^{\Sigma f_i},w^{ f_i})\tOtimes
  \can(v^{\Sigma f_j},w^{ f_j})&=t^{\langle S_i,S_j\rangle -\langle
    S_j,S_i\rangle }\can(v^{\Sigma f_i}+v^{\Sigma f_j},w^{f_i}+w^{f_j}).
 \end{align}
\end{Prop}
\begin{proof}
  We deduce from \ref{prop:dominant_decomposition} that there exists
  no $l$-dominant pair $(v,w^{f_i}+w^{f_j})$ such that $v>
  v^{f_i}+v^{f_j}$ or $v>v^{f_i}+v^{\Sigma f_j}$ or $v>v^{\Sigma f_i}+v^{ f_j}$ or $v> v^{\Sigma
    f_i}+v^{\Sigma f_j}$. Therefore, the products in the statements
  only consist of the leading terms in \eqref{eq:multiplication_can}.

We only check the first product. The verification for the other
products are similar. It is straightforward to check that
\begin{align*}
d((v^{f_i},w^{f_i}),(v^{f_j},w^{f_j}))&=v^{f_i}\cdot \sigma^*
w^{f_j}=v^{f_i}\cdot e_{S_j}+v^{f_i}\cdot e_{\Sigma
  S_j}\\
&=\dim\Hom(S_i,S_j)+\dim \Ext(S_i,S_j),  
\end{align*}
and similarly
\begin{align*}
d((v^{f_j},w^{f_j}),(v^{f_i},w^{f_i}))=\dim\Hom(S_j,S_i)+\dim \Ext(S_j,S_i).  
\end{align*}
The statement follows.
\end{proof}

The following relations have already been proved in
\cite{HernandezLeclerc11}. We compute them for completeness.
\begin{Prop}\label{prop:Serre_relation}
  (i) For any $i,j\in I$ such that $\Ext_{\C Q}^1(S_j,S_i)=\C$, we have
  \begin{align*}
    (\can(0,e_{\sigma S_i})^{\otimes 2})\otimes \can(0,e_{\sigma S_j})-(t+t^{-1})
    \can(0,e_{\sigma S_i})\otimes \can(0,e_{\sigma S_j})\otimes  \can(0,e_{\sigma S_i})\\+
    \can(0,e_{\sigma S_j})\otimes (\can(0,e_{\sigma})
      S_i)^{\otimes 2}=0
  \end{align*}
 (ii) For any $i,j\in I$ such that $\Ext_{\C Q}^1(S_i,S_j)=\C$, we have
  \begin{align*}
    (\can(0,e_{\sigma S_i})^{\otimes 2})\otimes \can(0,e_{\sigma S_j})-(t+t^{-1})
    \can(0,e_{\sigma S_i})\otimes \can(0,e_{\sigma S_j})\otimes  \can(0,e_{\sigma S_i})\\+
    \can(0,e_{\sigma S_j})\otimes( \can(0,e_{\sigma})
      S_i)^{\otimes 2}=0
  \end{align*}
 (iii) In $(R,\otimes)$, For any $i,j\in I$ such that $\Ext_{\C Q}^1(S_j,S_i)=0$, we have
  \begin{align*}
    [\can(0,e_{\sigma S_i}),\can(0,e_{\sigma S_j})]=0.
  \end{align*}
\end{Prop}
\begin{proof}
(i)(ii) 
Notice that $i\neq j$. Denote $\delta_{\tau S_j,S_i}$ by $\delta$ and $\langle
S_i,S_j\rangle_a$ by $\chi$. Then the situations (i) and (ii) correspond to the
cases  $(\delta,\chi)=(1,1)$ and $(\delta,\chi)=(0,-1)$ respectively. For any pairs $(v^1,w^1)$, $(v^2,w^2)$, we define $\braket{(v^1,w^1),(v^2,w^2)}_a$ to be $\braket{w^1,w^2}_a$. Let us denote 
  \begin{align*}
    w'&=e_{\sigma S_i}+e_{\sigma S_j},\\
 w&=2e_{\sigma S_i}+e_{\sigma
      S_j}.
\end{align*}

We need the following coefficients because the
multiplication $\tilde{\otimes}$ is replaced by the twisted multiplication $\otimes$
\begin{align*}
A&=t^{-\Hf\braket{(0,e_{\sigma S_i}),(0,e_{\sigma S_j})}_a},\\
B&=t^{-\Hf\braket{(e_{S_i},e_{\sigma S_i}),(0,e_{\sigma S_j})}_a},\\
C&=t^{-\Hf\braket{(0,e_{\sigma S_i}),(0,w')}_a},\\
D&=t^{-\Hf\braket{(e_{S_i},e_{\sigma S_i}),(0,w')}_a},\\
E&=t^{-\Hf\braket{(0,e_{\sigma S_i}),(e_{S_i},w')}_a}.
  \end{align*}
It follows that $A=B=C=D=E=t^{-\Hf\chi}$.

First compute the following bilinear forms
\begin{align*}
  d((e_{S_i},e_{\sigma S_i}),(0,e_{\sigma S_j}))=0,\\
  d((0,e_{\sigma S_j}),(e_{S_i},e_{\sigma S_i}))=\delta,\\
  d((e_{S_i},e_{\sigma S_i}),(0,w'))=1,\\
  d((0,w'),(e_{S_i},e_{\sigma S_i}))=\delta,\\
  d((0,e_{\sigma S_i}),(e_{S_i},w'))=0,\\
  d((e_{S_i},w'),(0,e_{\sigma S_i}))=1.
\end{align*}

Similar to the proof of Proposition \ref{prop:EF}, we have the
following decompositions
\begin{align*}
  \can(0,e_{\sigma S_i})\otimes \can(0,e_{\sigma S_j})&=A \can(0,w')+Bt^\delta\can(e_{S_i},w'),\\
  \can(0,e_{\sigma S_j})\otimes \can(0,e_{\sigma S_i})&=A^{-1} \can(0,w')+B^{-1}t^{-\delta}\can(e_{S_i},w'),\\
  \can(0,e_{\sigma S_i})\otimes \can(0,w')&=C \can(0,w)+D t^{\delta-1}\can(e_{S_i},w),\\
\can(0,w')\otimes \can(0,e_{\sigma S_i})&=C^{-1} \can(0,w)+D^{-1}t^{1-\delta}\can(e_{S_i},w),\\
\can(0,e_{\sigma S_i})\otimes \can(e_{S_i},w')&=Et\can(e_{S_i},w),\\
\can(e_{S_i},w')\otimes \can(0,e_{\sigma S_i})&=E^{-1}t^{-1}\can(e_{S_i},w).
\end{align*}
The Proposition follows from direct calculation.

(iii) The statement is obvious.
\end{proof}

\begin{proof}[{Proof the Theorem \ref{thm:isom}}]
(i) We replace $\tOtimes$ by
$\otimes$ in $R$. The relations in Propositions \ref{prop:EK}
\ref{prop:EF} \ref{prop:KK} now become 
\begin{align}
  \can(0,e_{\sigma S_i})\otimes
  \can(v^{f_j},w^{f_j})&=t^{a_{ij}}\can(v^{f_j},w^{f_j})\otimes
  \can(0,e_{\sigma S_i})\\
\can(0,e_{\sigma S_i})\otimes
    \can(v^{\Sigma f_j},w^{f_j})&=t^{-a_{ji}}\can(v^{\Sigma f_j},w^{ f_j})\otimes
    \can(0,e_{\sigma  S_i})\\ 
  \can(0,e_{\sigma\Sigma S_i})\otimes
  \can(v^{f_j},w^{f_j})&=t^{-a_{ji}}\can(v^{f_j},w^{f_j})\otimes
  \can(0,e_{\sigma \Sigma S_i})\\
  \can(0,e_{\sigma \Sigma S_i})\otimes
  \can(v^{\Sigma f_j},w^{ f_j})&=t^{a_{ij}}\can(v^{\Sigma
    f_j},w^{f_j})\otimes
  \can(0,e_{\sigma \Sigma S_i}).
\end{align}
\begin{align}\label{eq:EF_reduce}
  [\can(0,e_{\sigma S_i}),\can(0,e_{\sigma\Sigma
    S_j})]=\delta_{ij}(t-t^{-1})(\can(v^{f_i},w^{f_i})-\can(v^{\Sigma f_i},w^{\Sigma
    f_i})).
\end{align}
\begin{align}
  \can(v^{f_i},w^{f_i})\otimes
  \can(v^{f_j},w^{f_j})&=\can(v^{f_i}+v^{f_j},w^{f_i}+w^{f_j})\\
  \can(v^{f_i},w^{f_i})\otimes
  \can(v^{\Sigma f_j},w^{f_j})&=\can(v^{f_i}+v^{\Sigma f_j},w^{f_i}+w^{f_j})\\
  \can(v^{\Sigma f_i},w^{ f_i})\otimes
  \can(v^{\Sigma f_j},w^{ f_j})&=\can(v^{\Sigma f_i}+v^{\Sigma f_j},w^{f_i}+w^{f_j}).
 \end{align}

Notice that the relations in Proposition \ref{prop:Serre_relation}
remain unchanged.

Comparing the above relations with those of the Chevalley generators of $\tUt(\mfr{g})$, we
can define a surjective algebra homomorphism $\phi$ from $\tUt(\mfr{g})_{\Q(t^\Hf)}$ to the Grothendieck ring $(K_{\Q(t^\Hf)}^*,\otimes)$, such that
\begin{align*}
\phi (t^\Hf)&=t^\Hf,\\
  \phi (E_i)&=\frac{-t}{t^2-1}\can(0,e_{\sigma S_i}),\\
\phi(K_i)&=\can(v^{\Sigma f_i},w^{f_i}),\\
\phi (K_i')&=\can(v^{f_i},w^{f_i}),\\
 \phi (F_i)&= \frac{t}{t^2-1}\can(0,e_{\sigma \Sigma S_i}).
\end{align*}
 This map is an isomorphism by Theorem \ref{thm:isom_coordinate_ring}
 and Theorem \ref{thm:triangular_decomposition}. We define $\kappa=\phi^{-1}$.

(ii) Notice that, for any $i\in I$, $\can(v^{f_i},w^{f_i})\otimes \can(v^{\Sigma
  f_i},w^{f_i})=\can(v^{f_i}+v^{\Sigma f_i}, 2 w^{f_i})$ is a center
element in $(R,\otimes)$. The statement follows from (i).
\end{proof}

\begin{proof}[{Proof of Theorem \ref{thm:Lusztig_basis}}] 
The claim follows from \cite{HernandezLeclerc11} by Proposition
\ref{prop:same_N}. % We briefly recall the proof.

\end{proof}

\section{Comparison of products}
\label{sec:comparison}
To conclude the paper, we show that the twisted product $\otimes$ of
the Grothendieck ring $R^+$ agrees with the
non-commutative multiplication $*$ defined in
\cite{HernandezLeclerc11}\cite{Hernandez02} via the reduction from
$(v,w)$ to $w-C_q v$. Notice that the twisted products do not agree in general, and, usually, such reduction is impossible because Cartan elements are not center elements, 
\cf Example \ref{eg:compare_product}.

Recall that the restriction of the twisted product $\otimes$
on $R^+$ is determine
by the bilinear form $\cN$:
\begin{align}
  \cN(m^1,m^2)=d(m^2,m^1)+d(m^1,m^2)+\Hf\langle \Phi(w^2),\Phi(w^1)\rangle_a,
\end{align}
for any $m^1=(v^1,w^1),m^2=(v^2,w^2)\in \N^{\Ind\Rep Q-\Inj Q}\times
W^S$, where we use $\Inj Q$ to denote the injectives in
$\Ind\Rep Q$. On the other hand, the non-commutative multiplication in \cite{HernandezLeclerc11} for the
corresponding Grothendieck ring is determined by the bilinear form
$\mathscr{N}$ defined on $W^+\times W^+$. For the rest of this
section, we will show these two products agree by proving Proposition \ref{prop:same_N}.

\cite[Remark 3.3]{HernandezLeclerc11} should imply Proposition
\ref{prop:same_N} which is the
main result of this section. We give an alternative approach to this result by considering the
``lift'' of $\tw\in W^+$ into $l$-dominant pairs.

For any $N\in\Ind\Rep(Q)$. Denote $[N]=\sum N_i S_i$ the class of
$N$ in $K_0(\Rep Q)$. For simplicity, we denote $\Hom_{\cD^b(Q)}(\ ,\ )$ by $\Hom(\ ,\ )$.

\subsection{$l$-dominant pairs}
For our purpose, we want to lift any $l$-dominant $(0,\tilde{w})$,
$\tilde{w}\in W^+$, to an $l$-dominant pair $(v,w)\in V^+\times W^S$,
whose existence is guaranteed by \cite{LP_rep_alg}.

Inspired by Corollary 3.15(iii) in \cite{LP_rep_alg}, we associate to any $N\in\Ind\Rep(Q)$ the pair $\iota(N)=(\iota_V(N),\iota_W(N))$ defined by
\begin{align}
  \begin{split}
    \iota_W(N)&=\sum_i(N_i\cdot e_{\sigma S_i}),\\
\iota_V(N)&=\sum_{x\in(\Ind\Rep(Q)-\Inj
      Q)}(\dim\Hom(\tau^{-1}x,[N])-\dim\Hom(\tau^{-1}x,N))\cdot e_x.
  \end{split}
\end{align}
In fact, we can rewrite $\iota_V(e_{\sigma
  N})=\sum_{x\in\Ind\Rep(Q)}(\dim\Hom(\tau^{-1}x,[N])-\dim\Hom(\tau^{-1}x,N))\cdot
e_x$ by taking $\tau$ as the functor defined for $\cD^b(Q)$.

\begin{Eg}
Let us take the example of Figure \ref{fig:A_2_rep_space}. Then we
have
\begin{align*}
  \iota(S_1)&=(0,e_{\sigma S_1})\\
  \iota(S_2)&=(0,e_{\sigma S_2})\\
  \iota(P_2)&=(e_{S_1},e_{\sigma S_1}+e_{\sigma S_2}).
\end{align*}
\end{Eg}

\begin{Prop}
The pair $\iota(N)$ is $l$-dominant and we have $\iota_W(N)-C_q
\iota_V(N)=e_{\sigma N}$.
\end{Prop}
\begin{proof}
The claim should be a translation of the result of
\cite{LP_rep_alg} from repetitive algebras to representations of
$Q$. We give a straightforward proof here.

To simplify the notations, let us denote the pair $\iota(N)$ by $\iota=(\iota_V,\iota_W)$.
For any $x\in\Ind\Rep Q$, denote the AR-triangle in $\cD^b(Q)$ by $\tau
x\ra E\ra x$, where $E=\oplus_j E_j$, with each $E_j$ an
indecomposable in $\cD^b(Q)$. Then the $\sigma x$-component of $\iota_W-C_q \iota_V$ is given by
\begin{align*}
  (\iota_W-C_q \iota_V)_{\sigma x}&= (\iota_W)_{\sigma
    x}-(\iota_V)_{\tau x}-(\iota_V)_{x}+ \sum_j(\iota_V)_{E_j}.
\end{align*}
It suffices to verify the following equality:
\begin{align*}
  (\iota_W-C_q \iota_V)_{\sigma x}=\delta_{x,N}.
\end{align*}

We have
\begin{align*}
(\iota_W)_{\sigma
    x}&=\sum_i\delta_{x,S_i}N_i,\\
  (\iota_V)_{\tau x}&=\dim\Hom(x,[N])-\dim\Hom(x,N),\\
  (\iota_V)_{ x}&=\dim\Hom(\tau^{-1}x,[N])-\dim\Hom(\tau^{-1}x,N),\\
\sum_j(\iota_V)_{E_j}&=\dim\Hom(\tau^{-1}E,[N])-\sum_j\dim\Hom(\tau^{-1}E,N).
\end{align*}
Applying the contravariant functor $\Hom(\ ,N)$ to the AR-triangle $x\ra \tau^{-1}E\ra
\tau^{-1}x\ra \Sigma x$, we obtain a long exact sequence
\begin{align*}
  \Hom(x,\Sigma^{-1}N)&\xra{w^1}\Hom(\tau^{-1}x,N)\ra
  \Hom(\tau^{-1}E,N)\\
&\xra{w^3} \Hom(x,N)\xra{w^2}\Hom(\tau^{-1}x,\Sigma N).
\end{align*}
Notice that $\Hom(x,\Sigma^{-1}N)=0$ and consequently $w^1=0$. By using the
universal property of AR-triangles, we see $w^3$ is surjective if
$N\neq x$ and $\dim\Cok w^3=1$ if $N=x$. Therefore we obtain
$$\dim\Hom(\tau^{-1}x,N)-
  \dim\Hom(\tau^{-1}E,N)+\dim \Hom(x,N)=\delta_{x,N}.$$

By applying the functors $\Hom(\ ,S_i)$ for all $i\in I$ to this AR-triangle, we obtain
$$\dim\Hom(\tau^{-1}x,[N])-
  \dim\Hom(\tau^{-1}E,[N])+\dim
  \Hom(x,[N])=\sum_i\delta_{x,S_i}N_i.$$

Putting these results together, we obtain the desired equality.
\end{proof}

\subsection{Comparison of bilinear forms}

For any $M,N\in \Ind\Rep(Q)$, recall that the Euler form $\langle
M,N\rangle=\dim\Hom(M,N)-\dim\Hom(M,\Sigma N)$ only depends on the
class $[M]$, $[N]$. The symmetrized Euler form is given by $(M,N)=\langle
M,N\rangle +\langle N,M\rangle$.

\begin{Def}[$q$-degree order]
  For any $(i,a),(j,b)\in \Ind\Rep(Q)$, we can write $a=q^{\xi(i)+A}$,
  $b=q^{\xi(j)+B}$, for some $0\leq A,B\leq 2h$ such that
  $\xi(i)+A-\xi(j)-B<h$. If $\xi(i)+A>\xi(j)+B$, we say the $q$-degree
  of $(i,a)$ is higher (or larger) than that of $(j,b)$ and the
  $q$-degree of $(j,b)$ is lower (or smaller) than that of $(i,a)$.
\end{Def}

\begin{Eg}
  In Figure \ref{fig:A_2_rep_space}, the $q$-degree of $P_2$ is higher than that of $S_1$.
\end{Eg}

\begin{Prop}\label{prop:same_form}
  For any different objects $M,N\in \Ind\Rep(Q)$. Assume that the $q$-degree of $M$ is
  not higher than that of $N$, then we have
  \begin{align}\label{eq:same_form}
d(\iota(N),\iota( M))-d(\iota(M),\iota(N))+\Hf\langle N,M\rangle_a=\Hf(M,N).
  \end{align}
\end{Prop}
\begin{proof}

By definition, we have
\begin{align*}
  \Hf(M,N)-\Hf\langle N,M\rangle_a&=\langle M,N\rangle,\\
d(\iota(N),\iota(M))&=e_{\sigma N}\cdot \sigma^*
\iota_V(M)+\iota_V(N)\cdot \sigma^* \iota_W(M)\\
&= e_{\tau N}\cdot \iota_V(M)\cdot+\iota_V(N)\cdot \sum M_i e_{\sigma S_i},\\
d(\iota(M),\iota(N))&=e_{\tau M}\cdot \iota_V(N)+\iota_V(M)\cdot \sum
N_i e_{\sigma S_i}.
\end{align*}
So we should check
\begin{align}\label{eq:verify_form}
  e_{\tau N}\cdot \iota_V(M)+\iota_V(N)\cdot \sum M_i e_{\sigma
    S_i}-e_{\tau M}\cdot\iota_V(N)-\iota_V(M)\cdot\sum N_i e_{\sigma S_i}=\langle M,N\rangle
\end{align}
First assume that $N$ and $M$ are not projective. By using the definition of $\iota_V$, we have
\begin{align*}
\iota_V(M)\cdot e_{\tau N}&=\dim\Hom(N,[M])-\dim\Hom(N,M),\\
\iota_V(N)\cdot
\sum M_i e_{\sigma S_i}&=\dim\Hom(\tau^{-1}[M],[N])-\dim\Hom(\tau^{-1}[M],N)\\
&=\dim\Hom([N],\Sigma[M])-\dim\Hom(N,\Sigma[M]),\\
\iota_V(N)\cdot    e_{\tau M}&=\dim\Hom(M,[N])-\dim\Hom(M,N),\\
\iota_V(M)\cdot
\sum N_ie_{\sigma S_i}&=\dim\Hom(\tau^{-1}[N],[M])-\dim\Hom(\tau^{-1}[N],M)\\
&=\dim\Hom([M],\Sigma[N])-\dim\Hom(M,\Sigma[N]).
\end{align*}
If $N$ is projective, we have $e_{\tau N}\cdot\iota_V(M)=0$. On
the other hand, $\dim \Hom(N,[M])-\dim\Hom(N,M)$ vanishes. So the above
expression of $e_{\tau N}\cdot\iota_V(M)$ remains
effective. Similarly, the above expression of
$e_{\tau M}\cdot\iota_V(N)$ remains effective even if $M$ is
projective. So we can remove the projectivity assumption on $M$ and $N$.

Because the $q$-degree of $M$ is no larger than that of $N$, we have
$\Hom(N,M)=0$. The left hand side of \eqref{eq:verify_form} becomes
\begin{align*}
(
&\dim\Hom(N,[M])-\dim\Hom(N,\Sigma[M]))\\
&-(\dim\Hom(M,[N])-\dim\Hom(M,\Sigma[N]))\\
&+\dim\Hom([N],\Sigma[M])-\dim([M],\Sigma[N])+\dim\Hom(M,N)\\
&\qquad =\langle N,[M]\rangle-\langle M,[N]\rangle +\dim
\Hom([N],\Sigma[M])-\dim([M],\Sigma[N])\\&\qquad \qquad+\dim\Hom(M,N).
\end{align*}

We can replace $N$, $M$ by $[N]$ and $[M]$ respectively in the last
expression. Then, by using definition of $\langle\ ,\ \rangle$, the
last expression becomes
$$
\dim\Hom([N],[M])-\dim\Hom([M],[N])+\dim\Hom(M,N)=\dim\Hom(M,N).
$$

As the last step, $\dim\Hom(M,N)=\langle M,N\rangle$ because the $q$-degree of $M$ is no
larger than that of $N$.
\end{proof}

\begin{Prop}\label{prop:same_N}
  For any dominant pairs $m^1=(v^1,w^1)$, $m^2=(v^2,w^2)$ in
  $\N^{\Ind\Rep(Q)-\Inj Q}\times W^S$, we have
  \begin{align}
    \cN(m^1,m^2)=\Hf\mathscr{N}(w-C_qv^1,w-C_q v^2)\label{eq:same_N},
  \end{align}
where the form $\mathscr{N}$ defined on $W^+\times W^+$ is the bilinear form in \cite[(5)]{HernandezLeclerc11}.
\end{Prop}
\begin{proof}
  By Proposition 3.2 of \cite{HernandezLeclerc11}, the right hand side
  of \eqref{eq:same_form} is just $\Hf\mathscr{N}(e_{\sigma M},e_{\sigma
    N})$. Therefore, we have $\cN=\Hf\mathscr{N}$ in the situation of
  Proposition \ref{prop:same_form}. Then the claim holds true in
  general because $\cN$ and $\mathscr{N}$ are anti-symmetrized bilinear forms.
\end{proof}

\section*{Acknowledgements}
\label{sec:ack}

The author is indebted to Bernhard Keller for inviting him to Paris
where most part of this paper was done. He is grateful to Yoshiyuki
Kimura and Hiraku Nakajima for various discussions. He thanks David
Hernandez, Bernard Leclerc, You Qi, Ben Webster, and Mikhail Gorksy
for comments. He is grateful to the referee for many suggestions and remarks.

%%%%%%%%%%%%%%%%%%%%%%%%%%
%%                      Bibliography
%%%%%%%%%%%%%%%%%%%%%%%%%%

%%  \bibliography{  \langle  your bibdatabase \rangle }

% \bibliographystyle{amsalphaURL}%Can also use halpha.bst.
% % %\bibliographystyle{}
% \bibliography{referenceEprint}

\begin{thebibliography}{XXZ06}

\bibitem[Bri13]{Bridgeland:Hall}
Tom Bridgeland, \emph{Quantum groups via hall algebras of complexes}, Annals of
  Mathematics \textbf{177} (2013), 739--759, \href
  {http://arxiv.org/abs/1111.0745v1} {\path{arXiv:1111.0745v1}}.

\bibitem[FR12]{FangRosso12}
Xing Fang, and Marc Rosso, \emph{Multi-brace cotensor Hopf algebras and quantum groups}, 2012, \href {http://arxiv.org/abs/1210.3096}
  {\path{arXiv:1210.3096}}.


\bibitem[GLS13]{GeissLeclercSchroeer11}
Christof Gei\ss, Bernard Leclerc, and Jan Schr{\"o}er, \emph{Cluster structures
  on quantum coordinate rings}, Selecta Mathematica \textbf{19} (2013), no.~2,
  337--397, \href {http://arxiv.org/abs/1104.0531} {\path{arXiv:1104.0531}}.

\bibitem[Gor13]{Gorsky13}
Mikhail Gorsky, \emph{Semi-derived {H}all algebras and tilting invariance of
  {B}ridgeland-{H}all algebras}, 2013, \href {http://arxiv.org/abs/1303.5879}
  {\path{arXiv:1303.5879}}.

\bibitem[Her04]{Hernandez02}
David Hernandez, \emph{Algebraic approach to q,t-characters}, Advances in
  Mathematics \textbf{187} (2004), no.~1, 1--52, \href
  {http://arxiv.org/abs/math/0212257} {\path{arXiv:math/0212257}}, \href
  {http://dx.doi.org/10.1016/j.aim.2003.07.016}
  {\path{doi:10.1016/j.aim.2003.07.016}}.

\bibitem[Her04b]{Hernandez04}
\bysame, \emph{The $t$-analogs of $q$-characters at roots of unity for quantum affine algebras and beyond}, Journal of Algebra \textbf{279} (2004), no.~2, 514--547, \href
  {http://arxiv.org/abs/math/0305366} {\path{arXiv:math/0305366}}.


\bibitem[HL13]{HernandezLeclerc11}
David Hernandez and Bernard Leclerc, \emph{Quantum {G}rothendieck rings and
  derived {H}all algebras}, Journal f\"{u}r die reine und
  angewandte Mathematik(Crelle's journal), 2013, \href {http://arxiv.org/abs/1109.0862v1}
  {\path{arXiv:1109.0862v1}}.

\bibitem[Kap98]{Kapranov98}
M.~Kapranov, \emph{Heisenberg doubles and derived categories}, J. Algebra
  \textbf{202} (1998), no.~2, 712--744.

\bibitem[Kas91]{Kas:crystal}
M.~Kashiwara, \emph{On crystal bases of the {$Q$}-analogue of universal
  enveloping algebras}, Duke Math. J. \textbf{63} (1991), no.~2, 465--516,
  Available from: \url{http://dx.doi.org/10.1215/S0012-7094-91-06321-0}, \href
  {http://dx.doi.org/10.1215/S0012-7094-91-06321-0}
  {\path{doi:10.1215/S0012-7094-91-06321-0}}. \MR{1115118 (93b:17045)}

\bibitem[KS13]{KellerScherotzke2013}
Bernhard Keller and Sarah Scherotzke, \emph{Graded quiver varieties and derived
  categories}, Journal f\"{u}r die reine und
  angewandte Mathematik(Crelle's journal), 2013, \href {http://arxiv.org/abs/1303.2318}
  {\path{arXiv:1303.2318}}.


\bibitem[Kim12]{Kimura10}
Yoshiyuki Kimura, \emph{Quantum unipotent subgroup and dual canonical basis},
  Kyoto J. Math. \textbf{52} (2012), no.~2, 277--331, \href
  {http://arxiv.org/abs/1010.4242} {\path{arXiv:1010.4242}}, \href
  {http://dx.doi.org/10.1215/21562261-1550976}
  {\path{doi:10.1215/21562261-1550976}}.


\bibitem[KQ14]{KimuraQin11}
Yoshiyuki Kimura and Fan Qin, \emph{Graded quiver varieties, quantum cluster
  algebras and dual canonical basis}, Advances in Mathematics \textbf{262} (2014): 261-312, \href
  {http://arxiv.org/abs/1205.2066} {\path{arXiv:1205.2066}}.


\bibitem[KL09]{KhovanovLauda08}
Mikhail Khovanov and Aaron~D. Lauda, \emph{A diagrammatic approach to
  categorification of quantum groups {I}}, Represent. Theory \textbf{13}
  (2009), 309--347, \href {http://arxiv.org/abs/0803.4121v2}
  {\path{arXiv:0803.4121v2}}.

\bibitem[KL10]{KhovanovLauda08:III}
\bysame, \emph{A diagrammatic approach to categorification of quantum groups
  {III}}, Quantum Topology \textbf{1} (2010), no.~1, 1--92, \href
  {http://arxiv.org/abs/0807.3250} {\path{arXiv:0807.3250}}.


\bibitem[Lec04]{Leclerc04}
Bernard Leclerc, \emph{Dual canonical bases, quantum shuffles and
  q-characters}, Math. Z., \textbf{246} (2004), no. 4,  691--732.

\bibitem[LP13]{LP_rep_alg}
Bernard Leclerc and Pierre-Guy Plamondon, \emph{Nakajima varieties and
  repetitive algebras}, Publ. RIMS Kyoto Univ. \textbf{49} (2013), 531--561,
  \href {http://arxiv.org/abs/1208.3910} {\path{arXiv:1208.3910}}.

\bibitem[Lus90]{Lusztig90}
G.~Lusztig, \emph{Canonical bases arising from quantized enveloping algebras},
  J. Amer. Math. Soc. \textbf{3} (1990), no.~2, 447--498.

\bibitem[Lus91]{Lusztig91}
\bysame, \emph{Quivers, perverse sheaves, and quantized enveloping algebras},
  J. Amer. Math. Soc. \textbf{4} (1991), no.~2, 365--421.

\bibitem[Lus93]{Lus:intro} \bysame, \emph{Introduction to quantum
    groups}, Progress in Mathematics, vol. 110, Birkh\"{a}user Boston Inc., Boston, MA, 1993. MR 1227098 (94m:17016).

\bibitem[Nak01]{Nakajima01}
Hiraku Nakajima, \emph{Quiver varieties and finite-dimensional representations
  of quantum affine algebras}, J. Amer. Math. Soc. \textbf{14} (2001), no.~1,
  145--238 (electronic), \href {http://arxiv.org/abs/math/9912158}
  {\path{arXiv:math/9912158}}.

\bibitem[Nak04]{Nakajima04}
\bysame, \emph{Quiver varieties and {$t$}-analogs of {$q$}-characters of
  quantum affine algebras}, Ann. of Math. (2) \textbf{160} (2004), no.~3,
  1057--1097, \href {http://arxiv.org/abs/math/0105173v2}
  {\path{arXiv:math/0105173v2}}.

\bibitem[Nak11]{Nakajima09}
\bysame, \emph{Quiver varieties and cluster algebras}, Kyoto J. Math.
  \textbf{51} (2011), no.~1, 71--126, \href {http://arxiv.org/abs/0905.0002v5}
  {\path{arXiv:0905.0002v5}}.

\bibitem[PX97]{PengXiao97}
Liangang Peng and Jie Xiao, \emph{Root categories and simple {L}ie algebras},
  J. Algebra \textbf{198} (1997), no.~1, 19--56.

\bibitem[PX00]{PengXiao00}
\bysame, \emph{Triangulated categories and kac-moody algebras}, Invent. Math.
  \textbf{140} (2000), no.~3, 563--603.

\bibitem[Qin13]{Qin12}
Fan Qin, \emph{t-analog of q-characters, bases of quantum cluster algebras, and
  a correction technique}, International Mathematics Research Notices (2013),
  \href {http://arxiv.org/abs/1207.6604} {\path{arXiv:1207.6604}}.

\bibitem[Rin90]{Ringel90}
Claus~Michael Ringel, \emph{Hall algebras and quantum groups}, Invent. Math.
  \textbf{101} (1990), no.~1, 583--591.

\bibitem[Rou08]{Rouquier08}
Rapha\"el Rouquier, \emph{2-{K}ac-{M}oody algebras}, 2008, \href
  {http://arxiv.org/abs/0812.5023} {\path{arXiv:0812.5023}}.

\bibitem[Sch09]{Schiffmann:canonical_basis}
\bysame, \emph{Lectures on canonical and crystal bases of {H}all
  algebras}, 2009, \href {http://arxiv.org/abs/0910.4460} {\path{arXiv:0910.4460}}.

\bibitem[Sch06]{Schiffmann06}
Olivier Schiffmann, \emph{Lectures on {H}all algebras}, 2006, \href
  {http://arxiv.org/abs/math.RT/0611617} {\path{arXiv:math.RT/0611617}}.

\bibitem[VV03]{VaragnoloVasserot03}
M.~Varagnolo and E.~Vasserot, \emph{Perverse sheaves and quantum {G}rothendieck
  rings}, Studies in memory of {I}ssai {S}chur ({C}hevaleret/{R}ehovot, 2000),
  Progr. Math., vol. 210, Birkh\"auser Boston, Boston, MA, 2003, pp.~345--365,
  \href {http://arxiv.org/abs/math/0103182v3} {\path{arXiv:math/0103182v3}}.
  \MR{MR1985732 (2004d:17023)}

\bibitem[VV11]{VaragnoloVasserot09}
\bysame, \emph{Canonical bases and {KLR}-algebras}, J. Reine Angew. Math.
  \textbf{2011} (2011), no.~659, 67--100, \href
  {http://arxiv.org/abs/0901.3992} {\path{arXiv:0901.3992}}, \href
  {http://dx.doi.org/10.1515/crelle.2011.068}
  {\path{doi:10.1515/crelle.2011.068}}.

\bibitem[Web10]{Webster10}
Ben Webster, \emph{Knot invariants and higher representation theory {I}:
  diagrammatic and geometric categorification of tensor products}, 2010, \href
  {http://arxiv.org/abs/1001.2020} {\path{arXiv:1001.2020}}.

\bibitem[Web13]{Webster13}
\bysame, \emph{Knot invariants and higher representation theory}, 2013, \href
  {http://arxiv.org/abs/1309.3796} {\path{arXiv:1309.3796}}.

\bibitem[XXZ06]{XiaoXuZhang06}
Jie Xiao, Fan Xu, and Guanglian Zhang, \emph{Derived categories and {L}ie
  algebras}, 2006, \href {http://arxiv.org/abs/math/0604564}
  {\path{arXiv:math/0604564}}.

\end{thebibliography}
\def\cprime{$'$}
\providecommand{\bysame}{\leavevmode\hbox to3em{\hrulefill}\thinspace}
\providecommand{\MR}{\relax\ifhmode\unskip\space\fi MR }
\providecommand{\MRhref}[2]{%
  \href{http://www.ams.org/mathscinet-getitem?mr=#1}{#2}
}
\providecommand{\href}[2]{#2}

\end{document}